\documentclass[]{amsart}

\usepackage{amssymb}
\usepackage{amsmath}
\usepackage{amsthm}
\usepackage{amscd}

\long\def\comment#1\endcomment{}
\usepackage{color}

\newcommand{\base}{\operatorname{base}}

\newcommand\notpitchfork{\not\pitchfork}

\newcommand {\iv}{^{-1}}
\newcommand{\Bl}{{\mathrm{B}}}
\newcommand{\mn}{{\mathrm{Min}}}
\newcommand{\dv}{{\mathrm{div}}}

\newcommand{\Dv}{{\mathrm{Div}}}

\newcommand{\vf}{\varphi}
\newcommand {\nn}{\mathcal N} 
\newcommand{\onn}{\overline{{\mathcal N}}}

\def\<{\left\langle}
\def\>{\right\rangle}

\def\MM{\mathcal {M}}
\def\MS{\mathcal {M}(S)}
\def\CC{{\mathcal C}}
\def\CCS{{\mathcal C}(S)}
\def\MCG{\mathcal{MCG}}
\def\QQ{{\mathcal Q}}
\newcommand {\pgot}{{\mathfrak p}}

\newcommand{\dcs}{\dist_{\CC (S)}}

\newcommand{\me}{\medskip}

\newcommand {\q}{\mathfrak q} 
\newcommand {\pg}{\mathfrak g} 
\newcommand{\la}{\langle}
\newcommand{\ra}{\rangle}
\newcommand{\pcs}{\pi_{\CC (S)}}
\newcommand{\fh}{\mathfrak{h}}

\newcommand{\fk}{\mathfrak{k}}

\def\MCG{{\mathcal MCG}}
\def\stab{{\rm stab}}

\def\ban{\mathbf{AN}}
\def\bn{\mathbf{N}}

\newtheorem{thm}{Theorem}[section]
\newtheorem{prop}[thm]{Proposition}
\newtheorem{cor}[thm]{Corollary}

\newtheorem{qn}[thm]{Question}
\newtheorem{lem}[thm]{Lemma}

\newtheorem{exa}[thm]{Example}
\newtheorem{examples}[thm]{Examples}

\newtheorem{notation}[thm]{Notation}

\theoremstyle{definition} \newtheorem{defn}[thm]{Definition}

\theoremstyle{remark}
\newtheorem{rmk}[thm]{Remark}

\def\square{\hfill${\vcenter{\vbox{\hrule height.4pt \hbox{\vrule
width.4pt height7pt \kern7pt \vrule width.4pt} \hrule
height.4pt}}}$}

\newcommand{\tsh}[1]{\left\{\kern-.9ex\left\{#1\right\}\kern-.9ex\right\}}
\newcommand{\Tsh}[2]{\tsh{#2}_{#1}}

\def\R{{\mathbb R}}
\def\free{{\mathbb F}}
\def\Z{{\mathbb Z}}
\def\N{{\mathbb N}}

\def\CC{{\mathcal C}}
\def\calc{{\mathcal C}}

\def\H{{\mathcal H}}
\def\lll{\mathcal{L}}

\def\pp{{\mathcal P}}
\def\calp{{\mathcal P}}

\def\nn{{\mathcal N}}

\def\fg{{\mathfrak g}}
\def\fc{{\mathfrak c}}

\def\fh{{\mathfrak h}}

\def\MCG{\mathcal {MCG}}

\def\nbhd{{\nn}}

\def\om{\omega}
\def\coneomega{{\rm{Cone}}_{\omega}}

\def\con{{\rm{Cone}}_{\omega}}

\def\dist{{\rm{dist}}}
\def\cone{\coneomega}

\def\ulim{\lim_\omega}

\def\<{\langle}
\def\>{\rangle}

\long\def\Restate#1#2#3#4{
\medskip\par\noindent
{\bf #1 \ref{#2} #3} {\it #4}\par\medskip }

\definecolor{darkgreen}{cmyk}{1,0,1,.2}

\def\autfn{{\rm{Aut}}(F_{n})}

\def\outfn{{\rm{Out}}(F_{n})}

\newcommand\intersect\cap
\newcommand\infinity\infty
\newcommand\wt\widetilde

\newcommand\inject\hookrightarrow
\newcommand\union\cup

\newcommand{\co}{\colon\thinspace}

\newcommand\join\Lambda
\newcommand\cross\times

\newcommand\lub\vee
\newcommand\glb\wedge
\newcommand\B{\mathcal{B}}

\renewcommand\paragraph[1]{\medskip\textbf{#1} }

\begin{document}
\title{Divergence, thick groups, 
and short conjugators}
\author{Jason Behrstock}\thanks{The research of the first author was partially supported
as an Alfred P. Sloan Fellow and by
NSF grant \#DMS-1006219}
\address{Lehman College,
City University of New York,
U.S.A.}
\author{Cornelia Dru\c{t}u}\thanks{The research of the second author was supported in part by
the EPSRC grant ``Geometric and analytic aspects of infinite groups" and by the project ANR Blanc ANR-10-BLAN 0116, acronyme GGAA}
\address{Mathematical Institute,
24-29 St Giles,
Oxford OX1 3LB,
United Kingdom.}
\date{October 22, 2011}

\begin{abstract} In this paper we explore relationships between divergence and thick groups, and with the same techniques we estimate lengths of shortest conjugators.
    We produce examples, for every positive integer
    $n$, of $CAT(0)$ groups which are thick of
    order $n$ and with polynomial divergence of order $n+1$, both
    these phenomena are new. With respect to thickness, these examples
    show the non-triviality at each level of
    the thickness hierarchy defined by Behrstock--Dru\c{t}u--Mosher
    in \cite{BehrstockDrutuMosher:thick1}. With respect to divergence
    our examples resolve questions of Gromov \cite{Gromov:Asymptotic} and Gersten
     \cite{Gersten:divergence3} \cite{Gersten:divergence}
    (the divergence questions were also recently and independently answered by
    Macura \cite{Macura:polydiv}). We also provide
    general tools for obtaining both lower and upper bounds on the divergence of
    geodesics and spaces, and we give the definitive lower bound for Morse
    geodesics in the $CAT(0)$ spaces, generalizing earlier results of
    Kapovich--Leeb \cite{KapovichLeeb:3manifolds} and Bestvina--Fujiwara \cite{BestvinaFujiwara:symmetric}. In the final section, we turn to the question of bounding
    the length of the shortest conjugators in several
    interesting classes of groups. We
    obtain linear and quadratic bounds on such lengths for classes of groups
    including $3$--manifold groups and
    mapping class groups (the latter gives new proofs of
    corresponding results of Masur--Minsky
    in the pseudo-Anosov case \cite{MasurMinsky:complex2} and Tao
    in the reducible case \cite{JTao:conj}).

\end{abstract}

\maketitle

\section{Introduction}

A main purpose of this paper is to provide a connection between two invariants: the divergence and the order of thickness. The \emph{divergence} arose in the study of non-positively curved manifolds and metric spaces and roughly speaking it measures the spread of geodesics. More precisely, given two geodesic rays $r,r'$ with $r(0)=r'(0)$ their divergence is defined as a map
$\dv_{r,r'} \colon\R_+ \to \R_+$, where $\dv_{r,r'} (t)$ is the
infimum of the lengths of paths joining $r(t)$ to $r'(t)$ outside the
open ball centered at $r(0)$ and of radius $\lambda t$. Here $\lambda$ is a fixed parameter in $(0,1)$ whose choice turns out to be irrelevant for the order of the divergence.

The \emph{order} of divergence and of other functions, or an \emph{upper bound} for such an order, are meant in the sense of the usual relations of order and equivalence. Given a constant $C\geq 1$ and two
functions $\R_+\to \R_+\, $, $f$ and $g$, we write $f\preceq_C g$ if
$f(x) \le Cg(Cx+C)+Cx+C\; \mbox{ for all }x\in \R_+.$ This yields
an equivalence relation on the set of functions $\R_+\to
\R_+\, $ by setting $f\asymp_C g$ if and only if
$f\preceq_C g$ and $g\preceq_C f$.
When there is no risk of confusion we do not mention the constant $C$
and remove the corresponding subscript.

In symmetric spaces of non-compact type the order of the divergence of geodesic
rays is either exponential (when the rank is one) or linear (when the
rank is at least two).  This inspired an initial thought that in the
presence of
non-positive curvature the divergence must be either linear or exponential.
See \cite{Gromov:Asymptotic} for a discussion --- an explicit
statement of this conjecture
appears in $6.B_2$, subsection ``Geometry of $\partial_T$ and Morse
landscape at infinity,'' Example (h).  In particular, Gromov stated
an expectation that all pairs of geodesic rays in the
universal cover of a closed Riemannian manifold of non-positive
curvature diverge either linearly or exponentially \cite{Gromov:Asymptotic}.

As an aside, we note that without the hypothesis of non-positive curvature
the situation is more
complicated.  For instance, in nilpotent groups with left invariant metrics,
while the maximal rate of divergence of geodesics is linear
\cite{DrutuMozesSapir}, there exist geodesic rays that diverge
sublinearly \cite[Lemma 7.1]{Pauls:NilpotentCAT(0)}.

Gersten provided the first examples of $CAT(0)$
spaces whose divergence did not satisfy the linear/exponential
dichotomy and showed that such examples are closely tied to other
areas in mathematics. The first such example was a $CAT(0)$ space admitting a
cocompact action of the group  $F_2\rtimes_\vf \Z$
with $\vf( a) = ab\, ,\, \vf (b)=b\, $ \cite{Gersten:divergence}; the
space constructed by Gersten contains rays with quadratic divergence and he
proves that no two rays in this space diverge faster than
quadratically.
Extending that work, Gersten then used divergence to distinguish
classes of closed 
3-manifolds \cite{Gersten:divergence3}. Modulo the geometrization
conjecture, he proved that the divergence of a
3-manifold is either linear, quadratic, or exponential;
where quadratic divergence occurs precisely for graph manifolds and
exponential divergence occurs precisely when at least one geometric component is
hyperbolic. Gersten asked explicitly in \cite{Gersten:divergence3} which orders
of polynomial divergence were possible in a $CAT(0)$ group.

In a different direction, the authors of the present paper together
with L. Mosher \cite{BehrstockDrutuMosher:thick1}
introduced a geometric property  called \emph{thickness} which was
proved to hold for many interesting spaces. The definition is an
inductive one and, roughly speaking, characterizes a space as
\emph{thick of order~$n$} if \emph{it is a network of subsets} which
are each thick of order $n-1$, i.e., any two points in the space can
be connected by a chain of subsets thick of order $n-1$ with each
intersecting the next in an infinite diameter set.  The base level of
the induction, i.e., spaces defined to be thick of order 0,
are metric spaces with linear divergence.  The precise
definition of \emph{thick} is given in Section~\ref{sec:networks}.

The very structure of a network turns out to be well adapted to estimates on divergence. Indeed let $X$ be a geodesic metric space which is a  $(\tau , \eta)$--tight network with respect to a collection of subsets $\lll$, in the sense of Definition \ref{def:metricnetw}. Let $\delta $ be a number in $(0,1)$ and let $\gamma \geq 0$.  For every subset $L\in \lll$ let $\Dv_\gamma^L (n; \delta )$ be the divergence function for the tubular neighborhood of $L$ of radius $\tau$, with the induced metric (see Definition \ref{Div} for the notion of divergence function of a metric space); define the \emph{network divergence} of $X$ as
    $$\Dv_\gamma^\lll (n; \delta )=\sup_{L\in\lll}\Dv_\gamma^L (n;
    \delta )\, .$$

    The following holds.

\Restate{Theorem}{thickdivergence}{}{
  The divergence function in $X$ satisfies
    \begin{equation*}
	\Dv^{X}_\gamma (n; \delta ) \preceq_C n\, \Dv_\gamma^\lll (n;
	\delta ) \end{equation*} where the constant $C$ only depends on
	the constants $\tau , \eta , \delta$ and $\gamma$.}

Groups and spaces which are thick of order 0 or 1 both yield very
rich classes of examples, see e.g., \cite{BehrstockDrutuMosher:thick1},
\cite{BehrstockCharney}, \cite{BrockMasur:WPrelhyp},
\cite{DrutuSapir:TreeGraded}. In the present paper we give the first
constructions of groups which are thick of order greater than 1;
indeed, for every positive integer $n$ we produce infinitely many
quasi-isometry classes of groups which are thick of order $n$, as
explained in the following theorem.  Moreover, using a
close connection between order of thickness and order of divergence we
establish that the very same classes of examples have polynomial
divergence of degree $n+1$.
We note that the case $n=0$ of this theorem is trivial and the
case $n=1$ follows from the above mentioned results in
\cite{Gersten:divergence3} combined with results from
\cite{BehrstockDrutuMosher:thick1} and
\cite{BehrstockNeumann:qigraph}. The following is established in
Section~\ref{sec:highorderthick}.

\begin{thm}\label{thm:introcat0}
For every positive integer $n$ there exists an infinite family of
pairwise non-quasi-isometric finitely generated groups which are each:
\begin{enumerate}
\item CAT(0) groups;
\item thick of order n;
\item with divergence of order $n+1$.
\end{enumerate}
\end{thm}

Natasa Macura has recently given an independent construction of examples of CAT(0) groups
with divergence of order $n$ for all positive integers
\cite{Macura:polydiv}.  Our examples and Macura's both give
complete answers to the questions of Gromov \cite{Gromov:Asymptotic}
and Gersten \cite{Gersten:divergence3, Gersten:divergence} discussed above.

The upper bound on divergence in Theorem~\ref{thm:introcat0} will
follow from  Theorem~\ref{thickdivergence}. The lower bound both for divergence and for the order of thickness
is proved by exhibiting a bi-infinite geodesic with divergence
precisely $x^{n+1}$.

 It would be interesting to know if either in
general, or under some reasonable hypotheses, the order of thickness
and the divergence are directly correlated, i.e., can the order of
thickness be shown to provide a lower bound in addition to the upper
bound which in this paper we show holds in general. A homogeneous version of this question is:

\begin{qn}\label{quest:thickvsdiv}
If a group is thick of order $n$ must its
divergence be polynomial of degree exactly $n+1$?
\end{qn}

More specific questions on the possible orders of divergence include:

\begin{qn}\label{quest:cat0gap}
Are there examples of $CAT(0)$-groups whose
divergence is strictly between $x^{n}$ and $x^{n+1}$ for some $n$?
\end{qn}

\begin{qn}
What are the $\asymp$--equivalence classes of divergence functions of
$CAT(0)$-groups?
\end{qn}

Our inductive construction in Theorem \ref{thm:introcat0} can be made to yield infinitely many
quasi-isometry classes because the quasi-isometry type of the base is an invariant of the space,
and this base is a $CAT(0)$ $3$--dimensional graph manifold. According to the main
result of \cite{BehrstockNeumann:qigraph}
we have infinitely many quasi-isometry classes of  $3$--dimensional graph manifolds to choose from.

Another geometric feature relevant for divergence is the presence of
\emph{Morse quasi-geodesics}.  These are quasi-geodesics which
represent in some sense ``hyperbolic directions,'' namely, they
satisfy the Morse Lemma: any $(K,C)$--quasi-geodesic $\gamma$ with
endpoints on the given quasi-geodesic $\q$ is contained in a uniformly
bounded tubular neighborhood of $\q$.

Morse quasi-geodesics and their relationship to divergence are
studied in Section~\ref{sect:Morsediv}. One topic discussed there is the following natural refinement of
Question~\ref{quest:cat0gap}.

\begin{qn}
    If $X$ is a $CAT(0)$ space,
    can the divergence of a Morse geodesic be greater than $x^{n}$
    and less than $x^{n+1}$?
\end{qn}

The following theorem provides a negative answer for the case $n=1$;
its statement is the most general version of previous known results
which required extra assumptions such as periodicity of the geodesic
or properness of the space $X$, see \cite{KapovichLeeb:3manifolds} or
Proposition \ref{lhalf}, and also \cite{BestvinaFujiwara:symmetric}.

\Restate{Theorem}{cor2}{}{
Let $\q$ be a Morse quasi-geodesic in a $CAT(0)$ metric space $(X,
\dist )$.  Then the divergence of $X$ is $\succeq x^2$.
}

Further results and questions on the relation between Morse geodesics and divergence may be found in Section~\ref{sect:Morsediv}.

In Section~\ref{sec:conjug}, we study the question of finding
shortest conjugators for Morse elements, in $CAT(0)$ groups and in groups with ``(non-positive curvature)-like behavior''.
We generalize results from the $CAT(0)$ setting to Morse geodesics in
other groups. In that section we prove the following, which we then
apply to graph manifolds in Corollary~\ref{thm:conjgraphmfld}. Recall that an action of a group $G$ on a graph $X$ is called $l$--{\em
acylindrical} for some $l>0$ (or simply \emph{acylindrical}) if the
stabilizers in $G$ of pairs of points in $X$ at distance $\ge l$ are
finite of uniformly bounded sizes.  Recall also that in a finitely
generated group $G$, for a finite generating set $S$ which we often
do not explicitly mention, we denote by
$|g|_S$ or simply by $|g|$ the distance from $1$ to $g\in G\, $ in the
word metric corresponding to $S$.

\Restate{Theorem}{thm:conjacyltree}{}{
     Let $G$ be a group acting cocompactly and $l$-acylindrically on a simplicial tree $T$. For every $R>0$ and for a fixed word metric on $G$ let $f(R)$ denote the supremum of all diameters of intersections $\stab (a)\cap \nn_R( g\stab (b))$, where $a$ and $b$ are vertices in $T$ at distance at least $l$, and $g\in G$ is at distance $\leq R$ from $1$.

There exists a constant $K$ such that if
two loxodromic elements $u, v$ are conjugate in $G$ then there exists $g$
conjugating $u,v$ such that
$$
|g| \leq f(|u|+|v|+K)+|u|+|v|+2K\, .
$$
}
Note that in this theorem we cannot simply replace ``loxodromic'' by
``Morse,'' since there might exist Morse elements of
$G$ in the stabilizer of a vertex, e.g., this is the case if
$G$ is free and $T$ is the quotient by a free factor.

Two natural questions related to the above result can be asked.

\begin{qn}
Can Theorem \ref{thm:conjacyltree} be extended to actions that are not cocompact~?
\end{qn}

\begin{qn}
What are the possible values of the function $f(R)$ in Theorem \ref{thm:conjacyltree}~?
\end{qn}

As a consequence of Theorem~\ref{thm:conjacyltree} we obtain the following.

\begin{cor}\label{thm: introconjgraphmfld}
Let $M$ be a 3-dimensional prime manifold, and let $G$ be its
fundamental group. For every word metric on $G$ there exists a constant $K$ such that the following holds:

\begin{enumerate}
  \item if $u, v$ are two Morse elements conjugate in $G$ then there exists $g$
conjugating $u,v$ such that
$$
|g| \leq K(|u|+|v|)\, .
$$
  \item If $u, v$ are
two arbitrary elements conjugate in $G$ then there exists $g$
conjugating $u,v$ such that
$$
|g| \leq K(|u|+|v|)^2\, .
$$
\end{enumerate}

\end{cor}

We also give a new unified proof of the following theorem, first proved in the
pseudo-Anosov case by Masur--Minsky \cite[Theorem~7.2]{MasurMinsky:complex2} and later
extended to the reducible case by Tao \cite[Theorem~B]{JTao:conj}.

\Restate {Theorem}{thm:shortestgen}{}{
Given a surface $S$ and a finite generating set $F$ of its mapping class group $\MCG (S)$ there exists a constant $C$ depending only on $S$ and on $F$ such that for every two conjugate
elements of infinite order $u$ and $v$ in $\MCG (S)$ there exists $g\in \MCG (S)$ satisfying
$v=gug\iv$ and
$$
|g| \leq C \left( |u| + |v| \right)\, .
$$
}

\subsection*{Acknowledgments}
We thank L.~Mosher, P.~Papasoglu, A.~Sale, and A.~Sisto for useful conversations and corrections.
The first author would also like to thank the law firms of Orrick,
Herrington \& Sutcliffe and Disability Rights Advocates, and in
particular Ren\'e Kathawala and Elina Druker for working to stop
the New York City Department of Education from discriminating against children with
disabilities.

\section{General preliminaries}\label{sect:prelim}

We recall some standard definitions and establish our notation.

\me

We use the notation $\nn_R(A)$ for the {\em (open) $R$-neighborhood} of
a subset $A$ in a metric space $(X, \dist )$, i.e. $\nn_R(A)=\{x\in X:
\dist (x, A)<R\}$.  If $A=\{a\}$ then $\nn_R(A)=B(a,R)$
is the open $R$-ball centered at $a$.

We use the notation $\onn_R (A)$ and $\bar{B}(a,R)$ to designate the
corresponding {\em closed neighborhoods} and {\em closed balls}
defined by non-strict inequalities.

We make the convention that $B(a,R)$ and $\bar{B}(a,R)$ are the empty
set for $R<0$ and any $a\in X$. The terms ``neighborhood'' and ``ball'' will always mean an open  neighborhood, respectively, ball.

\begin{notation} Let $a>1$, $b,x,y$ be
positive real numbers. We write $ x\leq_{a,b} y$ if $$ x\le ay+b.$$

We write $x \approx_{a,b} y$ if $ x\leq_{a,b} y$ and
$y\leq_{a,b} x$.
\end{notation}

Consider two constants $L\geq 1$ and $C\geq 0$.

An $(L,C)$--\emph{coarse Lipschitz map} is a map $f:X\to Y$ of a metric space $X$ to a metric space $Y$ such that
$$
\dist (f (x), f(x'))\leq_{L,C} \dist (x,x' ),
\hbox{ for all } x, x' \in X.
$$

An $(L,C)$--\emph{quasi-isometric embedding} is a map $f:X\to Y$ that satisfies
$$
\dist (f (x), f(x'))\approx_{L,C} \dist (x,x' )\, ,
\hbox{ for all } x, x' \in X.
$$

If moreover $Y\subseteq \nn_C (f(X))$ the map $f$ is called a \emph{quasi-isometry}.

An $(L,C)$-{\em quasi-geodesic} is an $(L,C)$--quasi-isometric
embedding $\pgot\colon I\to X$, where $I$ is a connected subset of the
real line.  A \emph{sub-quasi-geodesic} of $\pgot$ is a restriction
$\pgot|_J\, ,$ where $J$ is a connected subset of $I$.

When $I=[a,\infty)$ we call both $\pgot$ and its image $\pgot (I)$ an
$(L,C)$-\emph{quasi-geodesic ray}. When $I=\mathbb{R}$ we call both $\pgot $ and its image
\emph{bi-infinite $(L,C)$-quasi-geodesic}.

We call $(L,0)$-quasi-isometries (quasi-geodesics)
$L$-\emph{bilipschitz maps (paths)}.

When the constants $L,C$ are irrelevant they are often not mentioned.

\bigskip

When considering divergence, it is often useful to consider a
particular type of metric space: the
tree-graded space, as these spaces, especially in their appearance as ultralimits, are
particularly relevant.  Recall the following definition from
\cite{DrutuSapir:TreeGraded}: a complete geodesic metric space $\free$
is {\em tree-graded with respect to a collection} $\pp$ of
closed geodesic subsets (called {\it{pieces}}) when the following two
properties are satisfied:

\begin{enumerate}

\item[($T_1$)] Every two different pieces have at most one common
point.

\item[($T_2$)] Every simple geodesic triangle in $\free$ is contained in one piece.
\end{enumerate}

\begin{lem}[Dru\c{t}u--Sapir  \cite{DrutuSapir:TreeGraded}]\label{tx}
Let ${\mathbb F}$ be a space which is tree-graded with respect to a collection of pieces
$\mathcal{P}$.

\begin{itemize}
\item[(1)] For every point $x\in \free$, the set $T_x$ of topological
arcs originating at $x$ and intersecting any piece in at most one point
is a complete real tree (possibly reduced to a point). Moreover if
$y\in T_x$ then $T_y=T_x$.
\item[(2)] Any topological arc joining two points in a piece is
contained in the same piece. Any topological arc joining two points
in a tree $T_x$ is contained in the same tree $T_x$.
\end{itemize}
\end{lem}

\begin{lem}[Dru\c{t}u--Sapir \cite{DrutuSapir:TreeGraded},
    Lemma 2.31]\label{cutting}
Let $X$ be a complete geodesic metric space containing at least
two points and let $\calc$ be a non-empty set of cut-points in
$X$. There exists a uniquely defined (maximal in an appropriate
sense) collection $\calp$ of subsets of $X$ such that
\begin{itemize}
\item $X$ is tree-graded with respect to $\pp$;
\item any piece in $\pp$ is either a singleton or a set
with no cut-point in~$\calc$.
\end{itemize}
Moreover the intersection of any two distinct pieces from $\calp$
is either empty or a point from $\calc$.
\end{lem}

\section{Divergence.}

Throughout this section $(X, \dist)$, or just $X$,
will denote a geodesic metric space.

\subsection{Equivalent definitions for divergence}

We recall the various definitions of the divergence and the fact
that under some mild conditions all these functions are equivalent.
The main reference for the first part of this section is \cite[$\S
3.1$]{DrutuMozesSapir}.

Consider two constants $0<\delta<1$ and $\gamma \geq 0$.

For an arbitrary triple of points $a,b,c\in X$ with $\dist(c,\{a,b\})=r>0$, define $\dv_{\gamma }(a,b,c;\delta)$ as the
infimum of the lengths
of paths
connecting $a, b$ and avoiding the ball $\Bl(c,\delta r -\gamma )$.

If no such path exists, define $\dv_{\gamma }(a,b,c;\delta)=\infty$.

\begin{defn}\label{Div}
The {\em divergence function} $\Dv^X_{\gamma }(n ,\delta)$ of the space $X$ is defined as the supremum of all
numbers $\dv_{\gamma }(a,b,c;\delta)$ with $\dist(a,b)\le n$. When there is no danger of confusion we drop the superscript
$X$.
\end{defn}

A particular type of divergence will be useful to obtain
lower bounds for the function $\Dv$ defined as above. More precisely,
let $\q$ be a bi-infinite quasi-geodesic in the space $X$, seen as a
map $\q \colon \R \to X \, $ satisfying the required two inequalities. We define \emph{the divergence of this quasi-geodesic} as
the function
$$
\Dv_\gamma^\q \colon(0, +\infty ) \to (0, +\infty), \; \; \Dv_\gamma^\q (r) = \dv_{\gamma }(\q (r)\, ,\, \q (-r)\, ,\, \q
(0) ;\delta)\, .
$$

Clearly for every bi-infinite quasi-geodesic $\q $ in a space $X$, $\Dv_\gamma^\q \preceq \Dv^X_\gamma \, $.
\me

In what follows we call a metric space $X$ \emph{proper} if all its
closed balls are compact.  We call it \emph{periodic}
if for fixed constants $L\geq 1$ and $C\geq 0$ the orbit of some ball
under the group of $(L,C)$--quasi-isometries covers $X$.

\me

A metric space is said \emph{to satisfy the hypothesis} (Hyp$_{\kappa ,L}$) for some $\kappa \geq 0$  and $L\geq 1$ if it
is one-ended,
proper, periodic, and every point is at distance less than $\kappa$ from a
bi-infinite $L$--biLipschitz path.

\begin{exa}
A Cayley graph of a finitely generated one-ended group satisfies the hypothesis $\left( \mathrm{Hyp}_{\frac12 ,
1}\right)$.
\end{exa}

\begin{lem}[Lemma 3.4 in \cite{DrutuMozesSapir}]\label{rem6}
Assume that $X$ satisfies \emph{(Hyp$_{\kappa ,L}$)} for some $\kappa \geq 0$ and $L\geq 1$. Then for $\delta_0=
\frac{1}{1+L^2}$ and every
$\gamma \geq 4 \kappa $ the function $\Dv_{\gamma }(n ,\delta_0)$ takes only finite values.
\end{lem}

A new divergence function, more restrictive as to the set of triples considered, is the following.

\begin{defn}\label{div}
Let $\lambda\ge 2$. The {\em small divergence function} $\dv_{\gamma }(n;\lambda,\delta)$ is the supremum of all numbers
$\dv_{\gamma
}(a,b,c;\delta)$ with  $0\le \dist(a,b)\le n$ and
\begin{equation} \label{lambda}
\lambda\dist(c,\{a,b\})\ge \dist(a,b).
\end{equation}
\end{defn}

\me

We define two more versions of divergence functions, with a further restriction on the choice of $c$. For
every pair of points $a,b\in X$, we choose and fix a geodesic
$[a,b]$ joining them such that if $x,y$ are points on a geodesic $[a,b]$ chosen to join $a,b$ the subgeodesic $[x,y]
\subseteq [a,b]$ is
chosen for $x,y$.

We say that a point $c$ is {\em between} $a$ and $b$ if $c$ is
on the fixed geodesic segment $[a,b]$.

We define $\Dv'_{\gamma }(n;\delta)$ and $\dv'_{\gamma }(n;\lambda,\delta)$ same as $\Dv_{\gamma }$ and $\dv_{\gamma }$ before,
but restricting
$c$
to the set of points between $a$ and $b$. Clearly  $\Dv_{\gamma }'(n;\delta)\le \Dv_{\gamma }(n;\delta)$ and
$\dv_{\gamma }'(n;\lambda,\delta)\le \dv_{\gamma }(n;\lambda,\delta)$ for every $\lambda,\delta$.

All these versions of divergence are now shown to be equivalent under
appropriate conditions.

\begin{prop}[Corollary 3.12 in \cite{DrutuMozesSapir}]\label{cdiv}

Let $X$ be a space satisfying the hypothesis \emph{(Hyp$_{\kappa
,L}$)} for some constants $\kappa \geq 0$ and $L\geq 1$, and let
$\delta_0=\frac{1}{1+L^2}$ and $\gamma_0 = 4\kappa $.

\me

\begin{enumerate}
    \item[(i)] Up to the equivalence relation $\asymp\, $, the functions
    $\dv_{\gamma }'(n;\lambda,\delta)$ and $\Dv_{\gamma }'(n;\delta)$
    with $\delta \leq \delta_0$ and $\gamma \geq \gamma_0\, $ are
    independent of the choice of geodesics $[a,b]$ for every pair of
    points $a,b$.

\me

\item[(ii)] For every $\delta \leq \delta_0$, $\gamma \ge \gamma_0$, and
$\lambda \geq 2$
\begin{equation*}\label{eq:equiv}
\Dv_\gamma (n; \delta ) \asymp \Dv_\gamma' (n; \delta ) \asymp
\dv_\gamma (n; \lambda ,\delta )\asymp \dv_\gamma' (n; \lambda ,\delta
)\, .
\end{equation*}

Moreover all the functions in this equation are independent of
$\delta \leq \delta_0$ and $\gamma \ge \gamma_0$ (up to the
equivalence relation $\asymp$).

\me

    \item[(iii)] The function $\Dv_{\gamma }(n;\delta)$ is equivalent to
    $\dv_{\gamma }'(n;2,\delta)$ as a function in $n$.  Thus in order
    to estimate $\Dv_{\gamma }(n,\delta)$ for $\delta\leq \delta_0$ it
    is enough to consider points $a,b,c$ where $c$ is the midpoint of
    a (fixed) geodesic segment connecting $a$ and $b$.
\end{enumerate}
\end{prop}

Proposition \ref{cdiv} implies that the $\asymp$--equivalence
class of the divergence function(s) is a quasi-isometry invariant in the class of metric spaces satisfying the hypothesis \emph{(Hyp$_{\kappa
,L}$)} for some constants $\kappa \geq 0$ and $L\geq 1$.

\me

The equivalent notions of divergence introduced previously are closely
related to the divergence as defined by S. Gersten in
\cite{Gersten:divergence} and \cite{Gersten:divergence3}.  We refer to
\cite{DrutuMozesSapir} for a detailed discussion.

There exists a close connection between the
linearity of divergence and the existence of global cut-points in
asymptotic cones;  see
\cite{Behrstock:asymptotic} for an early example and
\cite{DrutuMozesSapir} for a general theory.

\begin{defn}\label{ums}
A metric space $B$ is \emph{unconstricted} if the following properties hold:
\begin{itemize}
\item[(1)] for some constants $c, \lambda, \kappa$, every point in $B$
is at distance at most $c$ from a bi-infinite $(\lambda ,
\kappa)$--quasi-geodesic in $B$; \item[(2)] there exists an
ultrafilter $\omega$ and a sequence $d$ such that for every sequence
of observation points $b$, $\con (B,b,d)$ does not have cut-points.
\end{itemize}

If (2) is replaced by the condition that every asymptotic cone is
without cut-points then the space $B$ is called \emph{wide}.
\end{defn}

\begin{prop} [Proposition 1.1 in \cite{DrutuMozesSapir}]\label{lm02}
 Let $X$ be a geodesic metric space.
\begin{itemize}
\item[(i)] If there exists $\delta \in (0,1)$ and $\gamma \geq 0$ such that the function
$\Dv_{\gamma }(n;\delta)$ is bounded by a linear function then every
asymptotic cone of $X$ is without cut-points.

    \me

\item[(ii)] If $X$ is wide then for every $0<\delta
<\frac{1}{54}$ and every $\gamma \ge 0$, the function $\Dv_{\gamma
}(n;\delta)$ is bounded by a linear function.

    \me

\item[(iii)] Let $\fg \co \R \to X $ be a periodic geodesic. If $\fg$ has
superlinear divergence then in any asymptotic cone, $\cone(X)$,
for which the limit of $\fg$ is nonempty there exists a collection
of proper subsets of $\cone(X)$ with respect to which it is
tree-graded and the limit of
$\fg$ is a transversal geodesic.
\end{itemize}
\end{prop}

\begin{rmk}
\begin{enumerate}
  \item In \cite[Proposition 1.1]{DrutuMozesSapir}, ``wide'' means a
  geodesic metric space satisfying condition (2) only.  For this
  reason in statement (ii) of Proposition 1.1 in \cite{DrutuMozesSapir} it is assumed that $X$ is periodic.
  However that condition is only used to ensure that condition (1) in
  our definition of wideness is satisfied.

      \me

  \item In Proposition \ref{lm02}, (ii), the hypothesis that $X$ is
  wide cannot be replaced by the hypothesis that $X$ is unconstricted.
  Indeed in \cite{OOS} can be found examples of unconstricted groups
  with super-linear divergence.
\end{enumerate}
\end{rmk}


\subsection{Morse quasi-geodesics and divergence}\label{subsec:morsedefs}
Examples of groups with linear divergence include groups satisfying a law, groups with a central element of infinite
order, uniform lattices in higher rank products of symmetric spaces
and Euclidean buildings and some non-uniform lattices \cite{DrutuMozesSapir}. Conjecturally, all non-uniform
lattices in higher rank have linear divergence.

As shown by Proposition \ref{lm02}, super-linear divergence is equivalent to the existence of cut-points in at least one asymptotic cone. Nothing more consistent can be said on divergence in this very general setting. On the other hand, a stronger property than existence of cut-points allows in many situations to find better estimates on divergence. This property is the existence of Morse quasi-geodesics.  A bi-infinite quasi-geodesic $\q$ in $X$ is called
{\em Morse} if every $(L,C)$--quasi-geodesic with endpoints on $\q$ is
at bounded distance from $\q$ (the bound depends only on $L,C$).  In a
finitely generated group $G$ an element is called \emph{Morse} if it has infinite order and the
cyclic subgroup generated by it is a Morse quasi-geodesic.

Several important classes of groups contain Morse elements. Behrstock proved in \cite{Behrstock:asymptotic} that every
pseudo-Anosov element in a mapping class group is Morse.  Yael
Algom-Kfir proved the same thing for fully irreducible elements of the
outer automorphism group $Out(F_n)$ in \cite{Algom-Kfir}.
In a relatively hyperbolic group every non-parabolic element is
Morse (\cite{DrutuSapir:TreeGraded}, \cite{Osin:RelHyp}).

In \cite{DrutuMozesSapir} it is proved that in a finitely generated
group acting acylindrically on a simplicial tree or on a uniformly
locally finite hyperbolic graph, any loxodromic element is Morse.  An
action on a graph is called $l$-{\em acylindrical} for some $l>0$ if
stabilizers of pairs of points at distance $\ge l$ are finite of
uniformly bounded sizes.  At times the constant $l$ need not be
mentioned.  Note that a group acting by isometries on a simplicial
tree with unbounded orbits always contains loxodromic elements
\cite{Bowditch:tightgeod}.

Existence of Morse quasi-geodesics implies existence of cut-points in
all asymptotic cones.  The converse is only known to be true for
universal covers of non-positively curved compact non-flat de Rham
irreducible manifolds due to the following two results combined with Proposition \ref{lm02}, (iii).

\begin{thm}[\cite{Ballmann:rank}, \cite{Ballmann:book},
\cite{BurnsSpatzier}]
 Let M be a non-positively curved de Rham irreducible manifold with a
 group of isometries acting co-compactly.  Then either M is a higher
 rank symmetric space or M contains a periodic geodesic which does not
 bound a half-plane.
\end{thm}

\begin{prop}(\cite[Proposition 3.3]{KapovichLeeb:3manifolds})\label{lhalf}
Let $X$ be a locally compact, complete, simply
connected geodesic metric space which is locally $CAT(0)$. A
periodic geodesic $\pg$ in $X$ which does not bound a flat half-plane
satisfies
$$
\Dv^\pg (r) \succeq r^2\, .
$$
\end{prop}

The lower estimate on divergence that enters as a main ingredient in this converse is as important as the converse itself. In Section \ref{sect:Morsediv} we prove that the estimate in Proposition \ref{lhalf} holds in a considerably more general $CAT(0)$ setting, as well as for many of the examples of groups with Morse elements quoted above.

\section{Divergence and networks of spaces}\label{sec:networks}

\subsection{Tight Networks}

We strengthen the definitions of networks of subspaces and subgroups
from \cite{BehrstockDrutuMosher:thick1}, with a view towards
divergence estimate problems.

A subset $A$ in a metric space is called $C$--\emph{path connected} if any two points in $A$ can be
connected by a path in $\nbhd_C(A)\, $.  We say that $A$ is $(C,
L)$--\emph{quasi-convex} if any two points in $A$ can be connected
in $\nbhd_C(A)\, $ by a $(L ,L)$--quasi-geodesic. When $C=L$ we simply say that $A$ is $C$-quasi-convex.

\begin{defn}\textbf{(tight network of subspaces).}\label{def:metricnetw}

Given $\tau $ and $\eta $ two non-negative real numbers we say that a
metric space $X$ is a \textit{$(\tau , \eta )$--tight network with
respect to a collection~$\lll$ of subsets} or that $\lll$ \textit{forms a $(\tau , \eta )$--tight network inside} $X$ if every subset $L$ in
$\lll$ with the induced metric is $(\tau, \eta)$--quasi-convex, $X$ is
covered by $\tau$--neighborhoods of the sets $L\in \lll$, and the
following condition is satisfied: for any two elements $L,L'\in \lll$
and any point $x$ such that $B(x, 3\tau )$ intersects both $L$
and $L'$, there exists a sequence of length $n\leq \eta$
$$L_{1}=L,L_{2}, \ldots , L_{n-1},L_{n}=L'\, ,\; \mbox{ with
}L_{i}\in\lll$$ such that for all $1\leq i<n$, $\nbhd_\tau(L_{i})\cap
\nbhd_\tau(L_{i+1})$ is of infinite diameter, $\eta$--path connected
and it intersects $B(x, \eta )\, .$ We write $(\bn)$ to refer to the
above condition about arbitrary pairs of elements in $\lll$.
\end{defn}

When $G$ is a finitely generated group and $\lll=\H$ a
collection of undistorted subgroups the following strengthening of the
above definition is sometimes easier to verify.

\begin{defn}[\textbf{tight algebraic network of
subgroups}]\label{dgnetwork}
We say a finitely generated group $G$ is an $M$--\emph{tight algebraic
network with respect to $\H$} or that $\H$ \emph{forms an $M$--tight algebraic network inside} $G$, if $\H$ is a collection of
$M$--quasi-convex subgroups whose union generates a finite-index
subgroup of $G$ and for any two subgroups $H,H'\in\H$ there exists a
finite sequence $H=H_{1}, \ldots,H_{n}=H'$ of subgroups in $\H$ such
that for all $1\leq i<n$, the intersection $H_{i}\cap H_{i+1}$ is
infinite and $M$--path connected.  We write $(\ban)$ to refer to the
above condition about arbitrary pairs $H,H'\in\H$.
\end{defn}

A modification of the proof of
\cite[Proposition 5.3]{BehrstockDrutuMosher:thick1} yields the following.

\begin{prop}\label{inetw}
Let $\H$ be a collection of subgroups that forms a tight algebraic
network inside a finitely generated group $G$ and let $G_1$ be the finite-index
subgroup of $G$ generated by the subgroups in $\H$.  Then $G$ is a
tight network with respect to the collection of left cosets
$$
\lll = \{ gH\mid g\in G_1, H\in \H \}\, .
$$
\end{prop}

\proof Since cosets cover the subgroup $G_1$ and $G\subset \nn_\tau (G_1)$ for some $\tau >0$, it remains to prove $(\bn)$. By left translation we may assume that $x=1$; also, it clearly suffices to prove condition $(\bn)$ for $L=H$, and $L'=gH'$, where
$H,H'\in \H$ and $g\in G_1\cap \Bl (1, 3\tau
)$.

The argument in Proposition 5.3 in \cite{BehrstockDrutuMosher:thick1}
implies that for every such pair there exists a sequence
$L_1=H,L_2,..., L_n = gH'$ in $\lll$ composed of concatenations of
left translations of sequences as in $(\ban)$ and that all $L_i$ intersect $B(1, 3\tau )$.  In particular every
pair $L_i,L_{i+1}$ is a left translation of a pair of subgroups
as in ($\ban$) of the form $g'H_1,g'H_2$ with $H_{1}\cap
H_{2}$ infinite and $M$--path connected and $g'\in G_1$ closer to $1$ than $g$.

By \cite[Lemma~2.2]{MSW:QTTwo} the intersection $\nn_\tau (L_i) \cap
\nn_\tau (L_{i+1})$ is at finite Hausdorff distance from $g(H_1\cap
H_2)$, hence it is of infinite diameter and $\eta$--path connected for
$\eta$ large enough.\endproof

\medskip

The following yields a large family of examples.

\begin{prop}\label{networkgraphgroups} Let $G$ be a fundamental group
of a graph of groups where all the vertex groups are quasi-convex and
the edge groups are infinite.  Then $G$ is a tight network with
respect to the family of all left cosets of vertex groups.
Moreover, if the graph of groups is simply connected, then $G$ is a
tight algebraic network with respect to the family of vertex groups.
\end{prop}
\proof Since left cosets of a subgroup cover and the vertex sets are
quasi-convex by hypothesis, to show that $G$ is a tight network it
remains to verify property~$(\bn)$.  Fix two left cosets $L,L'$ of
vertex subgroups of $G$ and a point $x$ in $\nbhd_{3\tau} (L )\cap
\nbhd_{3\tau}(L')$.  By left translation we may assume that $x=1$.
In the Bass-Serre tree, $L$ and $L'$ correspond to vertices at
distance at most $n$ apart with $n\leq 6\tau$ and letting $L=L_{0},
L_{1},\ldots, L_{n}=L'$ be the sequence of left cosets corresponding
to the vertices in the shortest path in the tree from $L$ to $L'$, it
satisfies all the properties of $(\bn)$.  In particular all $L_i$
intersect $B(1, \eta )$ for $\eta $ large enough, because there are
finitely many possibilities for $L,L'$ (left cosets of vertex groups
intersecting $B(1, 3\tau )$) and therefore for $L_i$.

If the graph of groups is simply connected, then the vertex groups
generate $G$.  The remaining properties of tight algebraic network follow immediately.  \endproof

Tight networks define natural decompositions of geodesics, as described below.

\begin{lem}\label{lem:covergeod}
Let $X$ be a geodesic metric space and $\lll$ a collection of subsets
of $X$ such that $X = \bigcup_{L \in \lll} \nbhd_\tau(L)$.  Every geodesic $[x,y]$ contains a finite sequence of consecutive
points $x_0=x,x_1,x_2,...,x_{n-1}, x_{n}=y$ such that:

 \begin{enumerate}
   \item\label{cover1} for every $i\in \{ 0,1,...,n-1\}$ there exists
   $L_i\in \lll$ such that $x_i,x_{i+1}\in\nbhd_{3\tau} (L_i )\, $;
   \item\label{cover2} for every $i\in \{ 0,1,...,n-2\}$, $\dist
   (x_i,x_{i+1}) \geq \tau $.
 \end{enumerate}
 \end{lem}

\proof We inductively construct a sequence of consecutive points
$x_0=x, x_1, x_2, ..., x_n$ on $[x,y]$ such that for every $i\in \{
0,1,...,n-1\}$ there exists $L_i\in \lll$ with the property that
$x_i\in \nn_\tau (L_i)$ and $x_{i+1}$ is the farthest point from $x$
on $[x,y]$ contained in $\onn_{2\tau }(L_i)$.  Assume that we found
$x_0=x,x_1,...,x_k$.  Since $x_k \in X = \bigcup_{L \in \lll}
\nbhd_\tau(L)$ there exists $L_{k+1}\neq L_k$ such that $x_k \in
\nbhd_\tau(L_{k+1})$.  Pick $x_{k+1}$ to be the farthest point from
$x$ on $[x,y]$ contained in $\onn_{2\tau }(L_{k+1})$.  By our choice
of $x_{k}$ it follows that this process will terminate with $n\leq
\lceil \frac{\dist (x,y )}{\tau}\rceil$.  \endproof

\begin{lem}\label{lem:covergeod2}
Let $X$ be a geodesic metric space which is a $(\tau , \eta )$--tight
network with respect to a collection of subsets $\lll$, let $[x,y]$ be
a geodesic in $X$ and let $L,L'\in \lll$ be such that $x\in \nn_\tau
(L)$ and $y\in \nn_\tau (L')$.  There exists a finite sequence
of consecutive points $x_0=x,x_1,x_2,...,x_{n-1}, x_{n}=y$ on $[x,y]$
and a finite sequence of subsets $L_{j}\in\lll$, satisfying
$L_0=L,L_1,..., L_q=L'$ with $q\leq n\eta$ such that
 \begin{enumerate}
   \item\label{cover2.1} for every $i\in \{ 0,1,...,n-2\}$, $\dist
   (x_i,x_{i+1}) \geq \tau $;

   \smallskip

   \item\label{cover2.2} for every $j\in \{0,1,...q-1\}$ the
   intersection $\nn_\tau (L_j)\cap \nn_\tau (L_{j+1})$ is of infinite
   diameter and $\eta$--path connected;

       \smallskip

   \item\label{cover2.3} there exist $j_0=0< j_1<...<j_{n-1} <j_n=q$
   such that if $j_{i-1 }\leq j\leq j_{i}$ then $\nn_\tau (L_j)\cap
   \nn_\tau (L_{j+1})$ intersects $\Bl (x_i , \eta )\, $.
 \end{enumerate}
 \end{lem}

\proof For the geodesic $[x,y]$ consider a sequence
$x_0=x,x_1,x_2,...,x_{n-1}, x_{n}=y$ as in Lemma \ref{lem:covergeod},
and the sequence $L_0,L_1,..., L_{n-1}$ determined by the condition
Lemma~\ref{lem:covergeod} (\ref{cover1}), which can be taken with
$L_0=L$.

Property $(\bn)$ applied to each of the pairs $L_i,L_{i+1}$ with $i\in
\{ 0,1,...,n-2\}$, and to $L_{n-1}, L'$ provides a sequence
$L_0=L,L_1,..., L_q=L'$ with $q\leq n\eta$, such that for all $j\in
\{0,1,...q-1\}$ the intersection $\nn_\tau (L_j)\cap \nn_\tau
(L_{j+1})$ is of infinite diameter, $\eta$--path connected and it
intersects $\Bl (x_i , \eta )$ for some $0\leq i\leq n$.\endproof

The following shows that tight networks are a uniform version of the
\emph{networks} in \cite[Definition 5.1]{BehrstockDrutuMosher:thick1}.

\begin{cor}\label{cor:unifN2}
Let $X$ be a geodesic metric space which is a $(\tau , \eta )$--tight
network with respect to a collection of subsets $\lll$.

For every $M\geq 0$ there exists $R=R(M)$ such that for every $L,L'\in
\lll\, $ with $\nbhd_M (L )\cap \nbhd_M (L')\neq \emptyset$ and any
point $a\in\nbhd_{M} (L )\cap \nbhd_{M}(L')$ there exists a sequence,
$L_{1}=L,L_{2}, \ldots , L_{n-1},L_{n}=L'$, with $L_{i}\in\lll$ and
$n\leq R$ such that for all $1\leq i<n$, $\nbhd_\tau(L_{i})\cap
\nbhd_\tau(L_{i+1})$ is of infinite diameter, $\eta$--path connected,
and intersects $B(a, R)$.
\end{cor}

\proof Let $L,L'\in \lll$ be such that $\nbhd_M (L )\cap \nbhd_M
(L')\neq \emptyset$ and let $a$ be a point in the intersection.  Take
$x\in L$ and $y\in L'$ such that $\dist (x,a)<M$ and $\dist (y,a)<M$.
Lemma \ref{lem:covergeod2} applied to the geodesic $[x,y]$ yields a
sequence of $L_{i}$ for which $\nbhd_\tau (L_{i})\cap\nbhd_\tau
(L_{i+1})\cap \nbhd_{\eta}([x,y])\neq\emptyset$ and hence $\nbhd_\tau
(L_{i})\cap\nbhd_\tau (L_{i+1})\cap B(a,\eta+2M)\neq\emptyset$,
yielding the desired conclusion with $R(M)=\eta+2M$.  \endproof

\subsection{Network divergence}

We defined divergence
functions in Definition~\ref{Div}; for a network of spaces, we now
define an auxiliary function in order to bound the divergence of $X$.

\begin{defn}\label{def:netwdiv}
Let $X$ be a $(\tau , \eta )$--tight network with respect to a
collection~$\lll$ of subsets, let $\delta $ be a number in $(0,1)$ and
let $\gamma \geq 0$.  For every subset $L\in \lll$ we denote by
$\Dv_\gamma^L (n; \delta )$ the divergence function for $\nn_\tau (L)$
with the induced metric.

The \textit{network divergence of} $X$ is defined as
    $$\Dv_\gamma^\lll (n; \delta )=\sup_{L\in\lll}\Dv_\gamma^L (n;
    \delta )\, .$$
    \end{defn}

\begin{thm}\label{thickdivergence}
    Let $X$ be a geodesic metric space, let $\lll$ be a collection of subsets which forms a $(\tau , \eta
    )$--tight network inside $X$, let $\delta $ be a number in $(0,1)$ and let $\gamma \geq 0$. The divergence in $X$ satisfies
    \begin{equation}\label{divnetw}
	\Dv^{X}_\gamma (n; \delta ) \preceq_C n\, \Dv_\gamma^\lll (n;
	\delta ) \end{equation} where the constant $C$ only depends on
	the constants $\tau , \eta , \delta$ and $\gamma$.
 \end{thm}

\begin{proof} Let $a,b,c$ be three points such that $\dist (a,b) = n\,
$ and $\dist (c,\{ a,b\} )= r > \frac{\gamma}{\delta}\, $.  If the
ball $\Bl (c, \delta r -\gamma )$ does not intersect a geodesic
$[a,b]$ then $\dv_\gamma (a,b,c;\delta ) = n$.  Assume therefore that
$\Bl (c, \delta r -\gamma )$ intersects $[a,b]$.  This in particular
implies that $r\leq \delta r -\gamma + \frac{n}{2}$, whence $r\leq
\frac{n}{2(1-\delta )}\, $.

Lemma \ref{lem:covergeod2} applied to the geodesic $[a,b]$ implies the
existence of a finite sequence of points $x_0=a, x_1,..., x_k= b$ with
$k\leq \frac{n}{\tau }+2$, consecutive on the geodesic, and of a
finite sequence of subsets in $\lll$, $L_0=L,L_1,..., L_q=L'$ with
$q\leq k\eta\leq \eta \left( \frac{n}{\tau }+2 \right)$ such that:

 \begin{enumerate}
   \item\label{cv2.2} for every $j\in \{0,1,...q-1\}$ the intersection
   $\nn_\tau (L_j)\cap \nn_\tau (L_{j+1})$ is of infinite diameter and
   $\eta$--path connected;

       \smallskip

   \item\label{cv2.3} there exist $j_0=0< j_1<...<j_{n-1} <j_n=q$ such
   that if $j_{i-1 }\leq j\leq j_{i}$ then $\nn_\tau (L_j)\cap
   \nn_\tau (L_{j+1})$ intersects $\Bl (x_i , \eta )$.
 \end{enumerate}

Let $\sigma >0$ be large enough. Conditions on it will be added
later on.  Consider an arbitrary $j\in \{0,1,..., q-1\}\, $.  There
exists $i\in \{ 1,2,...,k-1\}\, $ such that $j_{i-1 }\leq j\leq
j_{i}$.  This implies that there exists a point $z_{j+1}$ in $\nn_\tau
(L_j)\cap \nn_\tau (L_{j+1})$ at distance at most $\eta$ from $x_i$.

Either $x_i$ is not in $\Bl (c, \sigma r )$ in which case we take
$y_{j+1} = x_i$, or $x_i$ is in $\Bl (c, \sigma r )$ and we proceed as
follows.  Since we are in a network, $\nn_\tau (L_j)\cap \nn_\tau
(L_{j+1})$ is of infinite diameter and thus contains an element $z_{j+1}'$
at distance $\geq \sigma r$ from $c$.  Since the intersection is
$\eta$--path connected there exists a path joining $z_{j+1}$ and
$z_{j+1}'$ in the $\eta$ neighborhood of the intersection, and one can
find on it a point at distance $\sigma r$ from $c$.  Thus, there exists
$y_{j+1}$ in $\nn_\tau (L_j)\cap \nn_\tau (L_{j+1})$ at distance
$\sigma r +O(1)$ from $c.$

If $x_0=a$ is in $\Bl (c, \sigma r )$ then we may find a point $y_0$
in $\nn_\tau (L_0)$ at distance $\sigma r +O(1)$ from $c.$ Likewise we
may have to replace $b$ by another point $y_{q+1}$.

We thus obtain a new sequence of points $y_0, y_1, ...., y_{q+1}$ all
at distance $\sigma r +O(1)$ from $c\, .$ For every $j\in \{
0,1,...,q\}$ the pair $y_j,y_{j+1}$ is inside $\nn_\tau (L_j)$.  If
$B(c, \delta r -\gamma )$ does not intersect $\nn_\tau (L_j)$ then
simply join $y_j,y_{j+1}$ by short geodesics to points in $L_j$ and
join those points by a $(\eta , \eta )$--quasi-geodesic in $\nn_\tau
(L_j)$. Otherwise, the intersection of $B(c, \delta r -\gamma )$ with
$\nn_\tau (L_j)$ contains a point, which we denote $c_j$. The ball
$B(c_j, 2 \delta r -\gamma )$ contains $B(c, \delta r -\gamma )$.
We choose $\sigma$ large enough so that $2 r \leq \dist (c_j ,
\{ y_j,y_{j+1}\} )$. Thus $\dist (c_j , \{ y_j,y_{j+1}\} ) \geq
\sigma r +O(1)$ allows us to join $y_j$ and $y_{j+1}$
outside the ball $B(c_j, 2 \delta r -\gamma )$ by a path of
length at most $\Dv^{L_j}_\gamma (2 \sigma r +O(1) ; \delta )\leq
\Dv^{L_j}_\gamma \left( \frac{\sigma}{1-\delta } n+O(1) ; \delta
\right)\, $.

The concatenation of all these curves gives a curve joining $a$ and
$b$ outside $B(c, \delta r -\gamma )$ and of length at most $2 \alpha
n + (q+1) \Dv^{\lll}_\gamma \left( \alpha n+O(1) ; \delta \right)$,
where $\alpha = \frac{\sigma}{1-\delta }\, $ and $q\leq \frac{\eta
}{\tau } n + 2\eta $.  Note that the first term stands for the lengths
of the geodesics joining $a$ and $y_0$, respectively $y_q$ and
$b$.\end{proof}

\begin{cor}\label{algthickdivergence}
    Let $G$ be a tight algebraic network with respect to the
    collection of subgroups $\mathcal{H}$. For every $\delta \in (0,1)$ and $\gamma \geq 0$,
    \begin{equation*}\label{divAnetw}
\Dv_\gamma^G (n; \delta) \preceq n\, \sup_{H\in
\mathcal{H}}\Dv_\gamma^H (n; \delta)\, .
    \end{equation*}
\end{cor}

\subsection{Thick spaces and groups}

In some sense, the subsets forming a network are building blocks and the ambient space is constructed out of them. By iterating this construction we obtain thick spaces, a notion introduced in
\cite{BehrstockDrutuMosher:thick1}.  The initial step in \cite{BehrstockDrutuMosher:thick1} (thick spaces of order zero) were taken to be unconstricted spaces. In this paper, we adapt the notion with a view to relate the order of
thickness to the order of the divergence function.  To this purpose, we introduce below \emph{strongly thick spaces} by taking
as initial step a subclass of unconstricted spaces, namely, wide spaces.

\begin{defn}\label{duunc}
A collection of metric spaces, $\B$, is \emph{uniformly wide} if:
\begin{itemize}
\item[(1)] for some positive real constants $\lambda, \kappa$, every point in every
space $B\in \B$ is at distance at most $\kappa$ from a bi-infinite
$(\lambda , \lambda)$--quasi-geodesic in $B$;

\item[(2)] for every sequence of spaces $(B_{i}, \dist_i)$ in $\B$, every ultrafilter
$\omega$, sequence of scaling constants $d= (d_i)$ and sequence of basepoints
$b=(b_i)$ with $b_i\in B_i$, the ultralimit $\ulim \left( B_{i}\, ,\,  b_i\, ,\,
\frac{1}{d_i}\, \dist_i \right)$ does not have cut-points.
\end{itemize}
\end{defn}

All the examples of unconstricted spaces listed in \cite[page
555]{BehrstockDrutuMosher:thick1} are in fact examples of uniformly
wide collections of metric spaces (with ``wide'' replacing
``unconstricted'' in Example 5).

The following uniform version of Proposition \ref{lm02} (ii) can be easily obtained by adapting the proof of \cite[Lemma 3.17 (ii)]{DrutuMozesSapir} and considering ultralimits of rescaled spaces in $\B$ instead of asymptotic cones.

\begin{prop}\label{lm03}
 Let $\B$ be a collection of uniformly wide metric spaces.

 For every $0<\delta <\frac{1}{54}$ and every $\gamma \ge 0$, the
 function $\sup_{B\in \B }\Dv^B_{\gamma }(n;\delta)$ is bounded by a
 linear function.
\end{prop}

\begin{defn}\textbf{(metric thickness and uniform thickness).}
\begin{itemize}
\item[\textbf{(M$_0$)}] A metric space is called \textit{strongly
thick of order zero} if it is wide.  A family of metric spaces is
\textit{uniformly strongly thick of order zero} if it is uniformly
wide.

\medskip

\item[\textbf{(M$_{n+1}$)}] Given $\tau \ge 0$ and $n\in\N$ we say
that a metric space $X$ is \textit{$(\tau , \eta )$--strongly thick of order at
most~$n+1$ with respect to a collection of subsets~$\lll$} if $X$ is a $(\tau
,\eta)$--tight network with respect to $\lll$, and moreover:

\smallskip

\begin{enumerate}
\item[$(\theta)$] the subsets in $\lll$ endowed with the restriction of the
metric on $X$ compose a family uniformly strongly thick of order at
most $n$.
\end{enumerate}

\medskip

\noindent Further, $X$ is said to be $(\tau ,\eta
)$--\emph{strongly thick of order~$n$ (with respect to~$\lll$)} if it is $(\tau ,\eta )$--\emph{strongly thick of
order at most~$n$ (with respect to~$\lll$)} and for no
choices of $\tau, \eta$ and $\lll$ is it strongly thick of order at
most~$n-1$.

When $\lll$, $\tau$, $\eta$ are
irrelevant, we say that $X$ is \emph{strongly thick of order (at most) $n$} or simply that $X$ is \emph{strongly thick}.

\medskip

\item[\textbf{(M$_{\rm{uniform}}$)}] A family $\left\{ X_i \mid i\in I
\right\}$ of metric spaces is \textit{uniformly strongly thick of
order at most }$n+1$ if the following hold.
\begin{itemize}
\item[$(\upsilon \theta_1)$] There exist $\tau >0$ and $\eta >0$ such
that every $X_i$ is a $(\tau , \eta)$--tight network with respect to a
collection~$\lll_i$ of subsets;

\item[$(\upsilon \theta_2)$]
$\bigcup_{i\in I}\lll_i$ is uniformly strongly thick of order at most
$n$, where each $L\in \lll_i$ is endowed with the induced metric.
\end{itemize}
\end{itemize}
\end{defn}

The order of strong thickness is a quasi-isometry invariant, c.f.,
\cite[Remark~7.2]{BehrstockDrutuMosher:thick1}.

For finitely generated groups and subgroups with word metrics a
stronger version of thickness can be defined.

\begin{defn}[\textbf{strong algebraic thickness}]\label{dgthick}
Consider a finitely generated group~$G$.
\begin{itemize}
\item[\textbf{(A$_0$)}] $G$ is called \textit{strongly algebraically thick of order zero} if it is wide.

\medskip

\item[\textbf{(A$_n$)}] Given $M>0$, a group $G$ is called $M$--\emph{strongly algebraically thick of order at most~$n+1$ with respect to a finite collection of subgroups $\H$}, if:

\medskip

\begin{itemize}
    \item $G$ is an $M$--tight algebraic network with respect to $\H$;

    \medskip

    \item all subgroups in $\H$ are strongly algebraically thick of
    order at most $n$.
\end{itemize}
\end{itemize}
\end{defn}
$G$ is said to be \emph{strongly algebraically thick of order~$n+1$
with respect to $\H$}, when $n$ is the smallest value for which this
statement holds.

\begin{rmk} The property of strong algebraic thickness does not
depend on the choice of the word metric on the group $G$. This
raised the question, asked in
\cite[Question 7.5]{BehrstockDrutuMosher:thick1}, whether
strong algebraic thickness is invariant by quasi-isometry.  The
following example, due to Alessandro Sisto, answers this question by
showing that algebraic thickness is not a quasi-isometry invariant.
Let $G$ be the fundamental group of a closed graph manifold whose associated
graph of groups consists of one vertex and one edge.
Since any element acting
hyperbolically on the Bass-Serre tree is a Morse element
\cite{DrutuMozesSapir}, it follows that any subgroup that contains
such an element has cut-points in all its asymptotic cones. Hence,
any subgroup of $G$  which is both quasi-convex
and wide (or even unconstricted) is contained in a conjugate of the
vertex group. As any finite set of subgroups contained, each of them, in some
conjugate of the vertex group generate an infinite index subgroup, it
follows that no collection of unconstricted subgroups can constitute a
tight algebraic network in $G$. Hence $G$ is not algebraically thick.
On the other hand, all fundamental groups of closed graph manifold are
quasi-isometric \cite{BehrstockNeumann:qigraph}
and some of them are algebraically thick, e.g., the graph
manifold built by gluing together two Seifert fibered spaces each
with one boundary component.
\end{rmk}

\begin{examples}\label{thickexamples}
    The following are some known examples of thick and
    algebraically thick spaces and groups:

\begin{enumerate}
  \item Mapping class groups of surfaces $S$ with complexity $\xi(S)=
  3\times\rm{ genus} + \#\mbox{\rm{(boundary components) }}-3 >1$ are
  strongly algebraically thick of order 1
  \cite{BehrstockDrutuMosher:thick1}, \cite{Behrstock:asymptotic};
  \item $\autfn$ and $\outfn$, for $n\geq 3$,  are strongly
  algebraically thick of order at most 1 with respect to a family of
  quasi-flats of dimension 2
  \cite{BehrstockDrutuMosher:thick1};
  $\outfn$ is strongly algebraically thick of
  order 1, this was announced in \cite{Hamenstadt:announcement}, see
  also  \cite{Algom-Kfir}
  \cite{Hamenstadt:Linesminimaout};
  \item various Artin groups are strongly algebraically thick
  of order at most one \cite[$\S 10$]{BehrstockDrutuMosher:thick1},
  right-angled Artin groups which are thick of order 1 are classified
  in \cite{BehrstockCharney};
  \item\label{graphgroups} graphs of groups with infinite edge groups
  and whose vertex groups are thick of order~$n$, are thick of order
  at most $n+1$, by Proposition~\ref{networkgraphgroups}.  In
  particular, the fundamental group $G=\pi_1(M)$ of a non-geometric
  graph manifold is strongly thick of order~$1$;
  \item for every surface $S$ of finite type with complexity $\xi(S)\geq 6$, the Teichm\"{u}ller space with the Weil-Petersson metric is strongly thick of order one with respect to a family of quasi-flats
  of dimension two \cite[$\S
  12$]{BehrstockDrutuMosher:thick1}, \cite{Behrstock:asymptotic}.
\end{enumerate}
\end{examples}

A connection between order of thickness and order of the divergence function can be easily established using Theorem~\ref{thickdivergence}.

\begin{cor}\label{thickimpliespolynom}

If a family $\B$ of metric spaces is uniformly strongly thick of order
at most $n$ then for every $0<\delta < \frac{1}{54}$ and every $\gamma
\geq 0$,

$$
\sup_{B\in \B} \Dv^B_\gamma (x;\delta )\preceq x^{n+1}\, .
$$

In particular, if a metric space $X$ is strongly thick of order at
most $n$, then for every $\delta$ and $\gamma $ as above:
$$
\Dv_\gamma (x;\delta )\preceq x^{n+1}\, .
$$
\end{cor}

\proof The statement follows by induction on $n$.  For $n=0$, since
wide spaces have linear divergence the result holds, see
Proposition~\ref{lm03}.  If the result holds for order~$n$, then
it follows immediately from Theorem~\ref{thickdivergence} that the
result holds for order $n+1$.\endproof

Corollary~\ref{thickimpliespolynom} yields upper estimates for
divergence functions of several spaces and groups. In the two
corollaries below, we record these estimates in two cases that are not in the literature. The estimates are sharp (for the upper bounds see
Section \ref{sect:Morsediv}), with one exception when the exact order of divergence is unknown.

\begin{cor}
    If $S$ is a compact oriented surface of genus $g$ and with $p$ boundary components such that $3g+p-3\geq 4$ and $(g,p) \neq (2,1)$ then the Weil-Petersson metric on the Teichm\"{u}ller space has at most
    quadratic divergence.  When $(g,p) = (2,1)$ the divergence is
    at most cubic.
\end{cor}

\proof It is an immediate consequence of Corollary~\ref{thickimpliespolynom} combined with \cite[Theorem~12.3]{BehrstockDrutuMosher:thick1} and \cite[Theorem~18]{BrockMasur:WPrelhyp}.\endproof

Quadratic lower bounds on the divergence of the Weil-Petersson metric
is implicit in the results of \cite{Behrstock:asymptotic}, see also Section
\ref{sect:Morsediv}.  The following question remains open, which if answered negatively would provide an interesting
quasi-isometry invariant differentiating the Weil-Petersson metric on
the Teichm\"{u}ller space of a surface of genus two with one boundary component from the other two Teichm\"{u}ller
spaces of surfaces of the same complexity (i.e., the four-punctured torus and the seven-punctured sphere).

\begin{qn}
    Does the Weil-Petersson metric on the Teichm\"{u}ller space of
    a surface of genus two with one boundary component have quadratic divergence?
\end{qn}

\begin{cor}\label{cor:lowerOut}
For $n\geq 3$ both $Aut (F_n)$ and $Out(F_n)$ have divergence at most
quadratic.
\end{cor}

This bound is sharp, for the lower bound see Corollary~\ref{cor:out}.

A natural question raised by Theorem~\ref{thickdivergence} and supported by all known examples, including the two above, is the
following:
\begin{qn}
    Is a group $G$ strongly algebraically thick of order $n$ if and only if it
    has polynomial divergence of degree $n+1$?
\end{qn}

\section{Higher order thickness and polynomial
divergence}\label{sec:highorderthick}

In this section we construct $CAT(0)$ groups that are strongly algebraically thick of order $n$ and with polynomial divergence of degree
$n+1$; this answers Gersten's question \cite{Gersten:divergence3} of
whether a $CAT(0)$ group can have polynomial divergence of degree 3
or greater (Macura recently provided a different construction of
examples \cite{Macura:polydiv}.)

We construct, by induction on $n$, a compact locally $CAT(0)$ space,
$M_n$, whose fundamental group $G_n =\pi_1 (M_n)$
is torsion-free. In Proposition~\ref{HOexamplethick}, we show that
$G_n$ is strongly algebraically  thick of order at most $n$.
In Proposition~\ref{HOexamplediv} we show that $M_{n}$ contains
a closed geodesic $\fg_n$ such that in the universal
cover $\widetilde{M}_n$ the lift $\widetilde{\fg}_n$ is Morse and has
divergence $\asymp x^{n+1}$.
Corollary~\ref{thickimpliespolynom} then
implies that $G_{n}$ is strongly thick of order exactly $n$.  Also, by
\cite[Theorem~3.2]{BehrstockNeumann:qigraph}, our construction of
$M_{1}$ can be chosen with fundamental group of any one of an
infinite family of pairwise non-quasi-isometric classes. In our
construction the quasi-isometry type of $M_{n}$ is an invariant of the
quasi-isometry type of $M_{n+1}$: hence our construction yields
infinitely many quasi-isometry types of groups. Thus this family of
groups will  yield Theorem~\ref{thm:introcat0}.

For $n=1$ take $M_1$ to be a CAT(0), non-geometric
graph manifold; these are easily constructed by taking a pair of
hyperbolic surfaces each with at least one boundary component, crossing each with a circle, and then gluing these two
$3$--manifolds together along a boundary torus by flipping the base and fiber directions. It
was proven by Gersten that these manifolds have quadratic divergence
\cite{Gersten:divergence3}. These groups are all
thick of order 1 \cite{BehrstockDrutuMosher:thick1},
algebraically thick examples are easy to produce by using
Example~\ref{thickexamples}(\ref{graphgroups})
and asking that the
corresponding graph of groups is simply connected (as in the explicit
example above). The remaining properties are easily verified.

Assume now that for a fixed integer $n\geq 1$ we have constructed a
compact locally $CAT(0)$ space $M_n$ with a closed geodesic $\fg_n$,
such that the lifts $\widetilde{\fg}_n$ in the universal cover have
divergence $\asymp x^{n+1}$; moreover such that the fundamental group,
$G_n=\pi_{1}(M_{n})$, is thick of order at most $n$. We
obtain  $M_{n+1}$ by gluing two isometric
copies of $M_n$ (denoted $M_n$ and $M'_n$)
by identifying the two copies of the closed geodesic $\fg_n$.

To check that $M_{n+1}$ is locally $CAT(0)$ we note that this
clearly holds in the neighborhood of each point $y$ not on $\fg_n$.
If $y\in \fg_n$ then any geodesic triangle with endpoints in $B(y,
\epsilon)$ 
is either contained
in one of the two copies of $M_n$ or two of its edges cross $\fg_n$. In
either of the cases it is easily checked that the triangle satisfies the $CAT(0)$
condition.

 It follows that $\widetilde{M}_{n+1}$ is a $CAT(0)$ space on which
 the fundamental group $G_{n+1}$ acts cocompactly.  The group
 $G_{n+1}$ is an amalgamated product of two copies of $G_n$ along the
 cyclic group $C_n$ generated by the element corresponding to $\fg_n$.
 We write this as $G_{n+1} = G_n \ast_{ C_n} G'_n$ (where $G_n$ and
 $G_{n}'$ are isomorphic).  The inductive
 hypothesis that $G_n$ is torsion-free implies that $G_{n+1}$ is
 torsion-free.  Let $T_n$ be the simplicial tree corresponding to this
 splitting.

 \begin{prop}\label{HOexamplethick}
 $G_{n+1}$ is strongly algebraically thick of order at most $n+1$.
 \end{prop}

\proof Since each of $G_{n}$ and $G_{n}'$ are thick of order at most $n$ and
intersect in an infinite cyclic group, it only remains to prove that $G_n$ and $G_n'$ are both quasi-convex in
$G_{n+1}$.  This is equivalent to proving that in
$\widetilde{M}_{n+1}$ the lifts $\widetilde{M}_n$ and
$\widetilde{M}_n'$ of $M_n$ and $M_n'$ respectively are quasi-convex.
Note that $\widetilde{M}_{n+1}$ is obtained by gluing all the
translates $G_n \widetilde{M}_n$ and $G_n \widetilde{M}_n'$ along
geodesics $G_n \widetilde{\fg}_n$.  In particular, the geodesics in
$G_n \widetilde{\fg}_n$ separate $\widetilde{M}_{n+1}$.

 We first prove that $\widetilde{\fg}_n$ is totally geodesic.  We prove by
 induction on $k$ that an arbitrary geodesic $[x,y]$ joining two
 points $x,y\in \widetilde{\fg}_n$ and crossing at most $k$ geodesics in $G_n
 \widetilde{\fg}_n$ must be contained in $\widetilde{\fg}_n$.  Here and in what follows
 when we say that a subset $A$ of $\widetilde{M}_{n+1}$ \emph{crosses} a geodesic
 $\overline{\fg}_n$ in $G_n \widetilde{\fg}_n \, $ we mean that $A$ intersects
 at least two connected components of $\widetilde{M}_{n+1}\setminus \overline{\fg}_n\, $.

 For $k=0$ the statement follows from the fact that $\widetilde{\fg}_n$ is totally geodesic
 both in $\widetilde{M}_n$ and in $\widetilde{M}_n'$ and that, since
 the metric on $\widetilde{M}_{n+1}$ locally coincides with the metric
 on $\widetilde{M}_n$ (respectively $\widetilde{M}_n'$) the length of
 a path contained in $\widetilde{M}_n$ (respectively
 $\widetilde{M}_n'$) is the same in that space as in
 $\widetilde{M}_{n+1}$.  Now assume that the statement is true for
 all integers less than $k$
 and consider an arbitrary geodesic $[x,y]$ with endpoints $x,y\in
 \widetilde{\fg}_n$ and crossing at most $k$ geodesics in $G_n \widetilde{\fg}_n$.
 Let
 $\fg_n'$ be a geodesic crossed by $[x,y]$ such that the corresponding
 edge in $T_n$ is at maximal distance from the edge corresponding to
 $\widetilde{\fg}_n$.  It follows that there exists $[a,b]$ subgeodesic of
 $[x,y]$ with endpoints on the geodesic $\fg_n'$ and not crossing any
 other geodesic in $G_n \widetilde{\fg}_n$.  Then it must be entirely
 contained in $\fg_n'$ according to the initial step for $k=0$.  Hence
 the geodesic $\fg_n'$ is not crossed and we can use the inductive
 hypothesis.

 We have thus proved that all geodesics in $G_n \widetilde{\fg}_n$ are totally
 geodesic in $\widetilde{M}_{n+1}\, ,$ in particular they are geodesics in
 $\widetilde{M}_{n+1}\, .$ From this it immediately follows that each of the
 subspaces in the orbits $G_n\widetilde{M}_n$ and
 $G_n\widetilde{M}_n'$ is totally geodesic.  \endproof

 \begin{prop}\label{HOexamplediv}
There exists a closed geodesic $\fg_{n+1}$ in $M_{n+1}$ such that in
the universal cover $\widetilde{M}_{n+1}$ the lift
$\widetilde{\fg}_{n+1}$ is Morse and has divergence $\asymp x^{n+2}$.
 \end{prop}

 \proof The group $G_{n+1}$ acts on the tree $T_n$ with quotient an
 edge; therefore there exists a loxodromic element $\gamma\in G_{n+1}$.
 The action is acylindrical, moreover the stabilizers of two distinct
 edges have trivial intersection.  Indeed consider two edges $e$ and
 $he$, with $h\in G_{n+1}\, $.  Assume that their stabilizers $C_n$
 and $hC_n h\iv$ intersect non-trivially.  Then they intersect in some
 finite index cyclic subgroup $C_n'\, $ of $C_n$.  In particular there
 exist two integers $r,s$ such that if $\gamma_n$ is the generator of
 $C_n$ then $\gamma_n^r= h \gamma_n^s h\iv \, $.  If $r\neq \pm s$
 then it can be easily proved that $C_n$ must be distorted in
 $G_{n+1}$, contradicting the previous lemma.  It follows that $r= \pm
 s$, and up to replacing $h$ by $h^2$ we may assume that $r=s \, $.
 It follows that $h$ is an element of infinite order in the center of
 $C_n'$, and this contradicts the fact that $C_n$ (and hence $C_n'$)
 is a Morse quasi-geodesic.

 By 
 Theorem~\cite[Theorem 4.1]{DrutuMozesSapir},
 since the cyclic subgroup
 $C=\la \gamma \ra$ acts acylindrically by
 isometries on a simplicial tree it is a Morse quasi-geodesic.  Consider a point
 $x\in \widetilde{M}_{n+1}\, $.  Since the map $G_{n+1} \to
 \widetilde{M}_{n+1}$ defined by $g\mapsto gx$ is a quasi-isometry, it
 follows that $C x$ is a Morse quasi-geodesic.  The sequence of
 geodesics $[\gamma^{-n}x , \gamma^n x]$ is contained in $\nbhd_M (Cx)
 $ for a fixed $M$, hence it has a subsequence converging to a
 bi-infinite geodesic $\pgot$ entirely contained in $\nbhd_M (Cx)\, .$
 For every $k\in \Z \, ,$ $\gamma^k \pgot$ is also inside $\nbhd_M
 (Cx)\, .$ It follows that the two bi-infinite geodesics $\pgot$ and
 $\gamma^k \pgot $ are at finite Hausdorff distance.  Since the
 function $t\mapsto \dist (\pgot (t) , \gamma^k \pgot (t))$ is convex
 positive and bounded it follows that it is constant, hence the two
 geodesics are parallel.  According to \cite{Bridson-Haefliger}, the
 set of bi-infinite geodesics parallel to $\pgot $ compose a set
 isometric to $\pgot \times K\, ,$ where $K$ is a convex subset.
 Since $C$ is a Morse quasi-geodesic, hence $\pgot$ is a Morse
 geodesic, it follows that $K$ must be bounded.  By possibly
 replacing it with a smaller set,  we may assume that $\pgot \times K$
 is invariant with respect to $C.$ If $b$ denotes the barycenter of
 $K$, then $\pgot \times \{ b\}$ is invariant with respect to $C\, .$ Take $\widetilde{\fg}_{n+1} = \pgot \times \{ b\}$ and $\fg_{n+1} =
 \widetilde{\fg}_{n+1}/C \, .$

 The only thing remaining to be proved is that the divergence of
 $\widetilde{\fg}_{n+1}$ is equivalent to $x^{n+2}\, $.  Consider a
 shortest path $\fc $ joining $\widetilde{\fg}_{n+1} (-x)$ and
 $\widetilde{\fg}_{n+1} (x)$ outside the ball $\Bl
 (\widetilde{\fg}_{n+1} (0), \delta x - \kappa )\, $.  In particular
 this path is at distance at least $\frac{\delta }{3} x$ from the
 geodesic $\widetilde{\fg}_{n+1}$ restricted to $\left[-\frac{\delta }{3} x\, ,\, \frac{\delta
 }{3} x \right]$ and it has to cross the same separating geodesics in
 $G_{n+1} \widetilde{\fg}_n$ and no more (if it is shortest).

 There exists a constant $M$ such that two separating geodesics
 crossed consecutively by $\widetilde{\fg}_{n+1}$ are at distance at
 most $M$ and such that the pair of points realizing the distance
 between two such consecutive geodesics is inside
 $\nbhd_M(\widetilde{\fg}_{n+1})$.  It follows that the number of
 separating geodesics in $G_{n+1} \widetilde{\fg}_n$ crossed by  $\widetilde{\fg}_{n+1}$ restricted to $\left[-\frac{\delta }{3} x\, ,\, \frac{\delta }{3} x
 \right]$ is $\asymp \frac{2\delta }{3} x\, $.  Let $\alpha$ and
 $\alpha'$ be the two intersection points of $\fc $ with two geodesics
 in $G_{n+1} \widetilde{\fg}_n$ crossed consecutively, $\fg$ and $\fg'$, and let
 $\fc'$ be the subpath of $\fc$ of endpoints $\alpha$ and $\alpha'$.
 Let $[a,b]$ be a geodesic which is the shortest path joining $\fg$
 and $\fg' \, $.  For $\delta'$ small enough we may assume that $\fc'$
 is outside $\nbhd_{\delta' x} ([a,b])$.  Let $\beta$ be the nearest
 point projection of $\alpha'$ onto $\fg \, $.  If $\dist
 (\alpha',\beta) =o(x)$ then an argument as in \cite[Proposition
 3.3]{KapovichLeeb:3manifolds} allows to find a flat strip bounded by
 $\fg$, a contradiction.  It follows that $\dist (\alpha',\beta) \geq
 \lambda x$ for some constant $\lambda >0$ independent of the point
 $\alpha'\, $.

 Both the geodesics $[a,b]$ and $[\beta,\alpha']$ make an angle of at
 least $\frac{\pi}{2}$ with $\fg$, since one of the endpoints of each
 is the nearest point projection on $\fg$ of the other endpoint.  This
 and the $CAT(0)$--property implies that $[\alpha ',\beta ] \cup [\beta
 ,a]\cup [a,b]$ is an $(L,0)$--quasi-geodesic, for large enough $L$
 (depending also on $\lambda$).  It follows that $[\alpha ',\beta ]
 \cup [\beta ,a]\cup [a,b]$ is contained in an $M$--neighborhood of
 $\fg' \, $.  In particular there exists $\beta'$ on $\fg'$ at
 distance $\asymp \lambda x$ such that $[\alpha', \beta' ] \subseteq
 \fg'$ has nearest point projection on $\fg$ at distance $O(1)$ from
 $\beta \, $.  Also $\beta$ is at distance $ O(1)$ from $a$, otherwise
 we would obtain again that some finite index subgroup of $C$ has a
 non-trivial element $h\not\in C$ in its center.  We choose a point
 $\mu $ on $\fg$ between $\alpha$ and $\beta$, at distance $\epsilon x$ from $\beta $ and a ball
 $\Bl (\mu , \frac{\epsilon}{10} x) \subset \nbhd_{\delta' x} ([a,b])$
 with $\varepsilon$ small enough. The ball $\Bl (\mu , \frac{\epsilon}{10} x)$ does not intersect the path $\fc'
 \cup [\alpha' , \beta ' ]\cup [\beta' , \beta ]\, $.  This and the
 fact that $\fg$ has divergence $\asymp x^{n+1}$ implies that the
 length of $\fc' \succeq x^{n+1} \, $.

 We conclude that the length of the path $\fc \succeq x^{n+2} \, $.\endproof

 \section{Morse quasi-geodesics and divergence}\label{sect:Morsediv}

 In this section we improve the result in Proposition \ref{lhalf} and generalize it to the utmost in the CAT(0) setting in Theorem \ref{cor2}. We also show that the quadratic lower bound on divergence occurs for many concrete examples of Morse elements in groups. This together with the estimate on divergence coming from the structure of thick metric space yields a divergence precisely quadratic for several groups and spaces (Corollary \ref{cor:out}).

\begin{lem}\label{lem:proj1}
Let $\q$ be a Morse quasi-geodesic in a metric space $(X, \dist )$.
Then for every $\lambda\in (0,1)$ and for every $M>0$ there exist $D
>0$ such that the following holds.  If $\fc$ is a sub-quasi-geodesic of
$\q$, $x$ and $y$ are points in $X$, $x'$ and $y'$ are points on $\fc$
minimizing the distance to $x$, respectively $y$, and $\dist (x',y')
\geq D$ while $\dist (x,x' )+ \dist (y',y)\leq M \dist (x',y')$ then
$\dist (x,y) \geq \lambda [\dist (x,x' )+ \dist (x' , y' )+ \dist
(y',y)]\, .$
\end{lem}

\proof Assume for a contradiction that there exists $\lambda \in
(0,1)$ and $M>0$ such that for every $D_n >0 $ there exist $\fc_n$
sub-quasi-geodesic of $\q$, $x_n,y_n\in X$ and $x_n', y_n'$ points on
$\fc_n$ minimizing the distance to $x_n$, respectively $y_n$ such that
$\dist (x_n',y_n') \geq D_n$, $\dist (x_n,x_n' )+ \dist (y_n',y_n)\leq
M \dist (x_n',y_n')$ while $\dist (x_n,y_n) \leq \lambda [\dist
(x_n,x_n' )+ \dist (x_n' , y_n' )+ \dist (y_n',y_n)]\, .$

We denote $\dist (x_n,x_n' )+ \dist (x_n' , y_n' )+ \dist (y_n',y_n)$
by $\delta_n$.  Note that by hypothesis $$
\dist (x_n' , y_n' ) \leq \delta_n \leq (M+1) \dist (x_n' , y_n' )\, .
$$

Consider $\coneomega (X, (x_n); (\delta_n))\, $.  Since $\q$ is a
Morse quasi-geodesic $\q_\om = \ulim (\q )$ is a biLipschitz path in
a transversal tree, $\fc_\om = \ulim (\fc_n )$ is a subpath in it, and
$x_\om' = \ulim (x_n')$ and $y_\om' = \ulim (y_n')$ are two points on
it.  Let $x_\om = \ulim (x_n)$ and $y_\om = \ulim (y_n)$.  Any
geodesics $[x_\om , x_\om']$ and $[y_\om , y_\om']$ intersect
$\fc_\om$ only in an endpoint.  Let $[x_\om'', x_\om']$ be the
intersection of the geodesic $[x_\om , x_\om']$ with the transversal
tree containing $\q_\om$; let $[y_\om'', y_\om']$ be defined likewise.
The union $[x_\om'', x_\om'] \cup [x_\om', y_\om']\cup [y_\om',
y_\om'']$ is a geodesic in a transversal tree, since it is a
concatenation of three arcs such that $[x_\om', y_\om']$ does not
reduce to a point, and it intersects its predecessor in $x_\om'$ and
its successor in $y_\om'$.

This, the fact that transversal trees can always be added to the list
of pieces in a tree-graded space \cite[Remark
2.27]{DrutuSapir:TreeGraded}, and Lemma 2.28 in
\cite{DrutuSapir:TreeGraded} imply that $[x_\om, x_\om'] \cup [x_\om',
y_\om']\cup [y_\om', y_\om]$ is a geodesic.  In particular $\dist (x_n
, y_n ) = \delta_n + o(\delta _n )$.  This contradicts the fact that
$\dist (x_n , y_n ) \leq \lambda \delta_n\, .$\endproof

\begin{lem}\label{lem:proj2}
Let $\q$ be a Morse quasi-geodesic in a $CAT(0)$ metric space $(X,
\dist )$.  Then for every $\lambda\in (0,1)$ there exists $D >0$ such
that the following holds.  If $\fc$ is a sub-quasi-geodesic of $\q$,
$x$ and $y$ are points in $X$, $x'$ and $y'$ are points on $\fc$
minimizing the distance to $x$, respectively $y$, and $\dist (x',y')
\geq D$ then $\dist (x,y) \geq \lambda [\dist (x,x' )+ \dist (x' , y'
)+ \dist (y',y)]\, .$
\end{lem}

\proof We argue by contradiction and assume that for some
$\lambda$ there exist sequences of sub-quasi-geodesics $\fc_n$ of
$\q_n$ and of pairs of points $x_n, y_n$ such that for some points
$x_n'$ and $y_n'$ on $\fc_n$ minimizing the distance to $x_n$,
respectively $y_n$ we have that $\dist (x_n',y_n') \geq n$ while
$\dist (x_n,y_n) \leq \lambda [\dist (x_n,x_n' )+ \dist (x_n' , y_n'
)+ \dist (y_n',y_n)]\, .$ In what follows we fix some geodesics $[x_n,
x_n']$ and $[y_n,y_n']$, and for every $u_n \in [x_n, x_n']$ and $v_n
\in [y_n,y_n']$ we introduce the notation
$$
\delta(u_n, v_n):=\dist (u_n,x_n' )+ \dist (x_n' , y_n' )+
\dist (y_n',v_n)\, .
$$

We will break the argument into two cases.

{\bf Case (i):} Assume $\lim_\om \frac{\delta (x_n , y_n )}{\dist (x_n' , y_n' )
} < \infty\, $.

Then in $\coneomega (X; (x_n'), \left(\dist (x_n' , y_n' ) \right))$,
the limit $\fc_\omega = \ulim \fc_n$ is a subarc in the transversal
line $\q_\omega = \ulim \q_n$ containing the two points $x_\om'$ and
$y_\om'$ distance $1$ apart, which points are the nearest points in
$\fc_\om$ to the points $x_\om$ and $y_\om \, $.

With an argument as in the proof of the previous lemma we obtain that
$[x_\om , x_\om'] \cup [x_\om', y_\om'] \cup [y_\om' , y_\om]$ is a
geodesic, in particular $\dist (x_n, y_n) = \delta (x_n, y_n) +
o(\dist (x_n' , y_n' ))$.  This contradicts the fact that $\dist
(x_n,y_n) \leq \lambda \delta (x_n, y_n)\, .$

{\bf Case (ii):}  Assume $\lim_\om \frac{\delta (x_n , y_n )}{\dist (x_n' ,
y_n' ) } = \infty\, $, or equivalently that
$$
\lim_\om \frac{\dist (x_n , x_n') + \dist (y_n, y_n')}{\dist (x_n' ,
y_n' ) } = \infty\, .
$$

We consider the parametrization proportional to the arc-length $\fg_x
\colon [0,1] \to [x_n, x_n']$ sending $0$ to $x_n'$ and $1$ to
$x_n$.  Similarly define
$\fg_y \colon[0,1] \to [y_n, y_n']$.

Fix $\lambda' \in (\lambda , 1)$ and for every $n$ consider the
maximal $t_n\in
[0,1]$ such that
$$
\dist (\fg_x (t_n), \fg_y (t_n )) \geq \lambda' \delta (\fg_x (t_n),
\fg_y (t_n ))\, .
$$

Clearly $t_n <1$ and from the continuity of the two sides of the inequality above and the maximality of $t_n$ we deduce that
\begin{equation}\label{eq:tn}
\dist (\fg_x (t_n), \fg_y (t_n )) = \lambda' \delta (\fg_x (t_n),
\fg_y (t_n ))\, .
\end{equation}

Using the convexity of the distance we have:
$$
\lambda' \delta (\fg_x (t_n), \fg_y (t_n )) = \dist (\fg_x (t_n),
\fg_y (t_n ))\leq (1-t_n)\dist (x_n' , y_n' )+ t_n \dist (x_n, y_n)
$$
$$
\leq (1-t_n)\dist (x_n' , y_n' )+ t_n \lambda \delta (x_n,y_n)\leq
\dist (x_n' , y_n' ) + \lambda \delta (\fg_x (t_n), \fg_y (t_n ))\, .
$$

Whence it follows that
$$
(\lambda' - \lambda) \delta (\fg_x (t_n), \fg_y (t_n )) \leq \dist
(x_n' , y_n' )\, .
$$

In particular $\ulim \frac{\delta (\fg_x (t_n), \fg_y (t_n ))}{\dist
(x_n' , y_n' ) }\leq \frac{1}{\lambda' - \lambda } < \infty\, $, whence
$$
\ulim \frac{\dist (\fg_x (t_n), x_n')+ \dist (y_n' , \fg_y (t_n ))}{\dist
(x_n' , y_n' ) } < \infty
$$

If the above limit is zero then since
$$
\dist (x_n' , y_n') -[\dist (\fg_x (t_n), x_n')+ \dist (y_n' , \fg_y (t_n ))]\leq \dist (\fg_x (t_n), \fg_y (t_n )) \leq \delta (\fg_x (t_n), \fg_y (t_n ))
$$ it follows that $\ulim \frac{\delta (\fg_x (t_n), \fg_y (t_n ))}{\dist
(x_n' , y_n' ) } = \ulim \frac{\dist (\fg_x (t_n), \fg_y (t_n ))}{\dist
(x_n' , y_n' ) } =1$.

Thus if in equation~\eqref{eq:tn} we divide by $\dist (x_n' , y_n' )$ and
take the $\omega$--limits we obtain $1=\lambda'$, a contradiction.

We conclude that
$$
0<\ulim \frac{\dist (\fg_x (t_n), x_n')+ \dist (y_n' , \fg_y (t_n ))}{\dist
(x_n' , y_n' ) } < \infty \, .
$$

In $\coneomega (X; (x_n'), \left(\dist (x_n' , y_n' ) \right))$, we
again have that $\fc_\omega = \ulim \fc_n$ is a subarc in the transversal
line $\q_\omega = \ulim \q_n$ containing the two points $x_\om'$ and
$y_\om'$ which are distance $1$ apart.  The limits $\ulim [x_n, x_n']$ and
$\ulim [y_n, y_n']$ are either two rays intersecting $\fc_\omega$ only
in their origin or only one ray like this and one geodesic segment
(possibly trivial) intersecting $\fc_\om$ only in one point.  The
limits of the sequences of points $\fg_x (t_n)$ and $\fg_y (t_n)$
are respectively on each of the two rays (or on the ray and the
segment), in particular it follows that $\dist (\fg_x (t_n), \fg_y
(t_n )) = \delta (\fg_x (t_n), \fg_y (t_n )) + o(\dist (x_n', y_n'))$.

This and equality \eqref{eq:tn} yield a contradiction.\endproof

\begin{lem}\label{lem1}
Let $\q$ be a Morse quasi-geodesic in a $CAT(0)$ metric space $(X,
\dist )$.  There exists a constant $D_0$ such that if $\fc$ is a
sub-quasi-geodesic of $\q$ and two points $x,y\in X$ are such that both
$\dist (x, \fc)$ and $\dist (y, \fc )$ are strictly larger than $\dist
(x,y)$ and $x'$ and $y'$ are points on $\fc$ minimizing the distance
to $x$, respectively $y$, then $\dist (x',y') \leq D_0\, $.
\end{lem}

\proof According to Lemma \ref{lem:proj1} there exists $D_0$ such that
if $\fc$ is a sub-quasi-geodesic of $\q$, $x,y\in X$ and $x'$ and $y'$
are points on $\fc$ minimizing the distance to $x$, respectively $y$,
$\dist (x,y) < \frac{1}{2} [\dist (x,x') + \dist (x',y') + \dist (y',
y) ]$ implies $\dist (x',y') \leq D_0$.\endproof

\begin{thm}\label{cor2}
Let $\q$ be a Morse quasi-geodesic in a $CAT(0)$ metric space $(X,
\dist )$.  Then the divergence $\Dv^\q \succeq x^2$.
\end{thm}

\proof Let $a=\q(-r)$ and $b=\q (r)$ and let $\pgot $ be a path
joining $a$ and $b$ outside $\Bl (\q (0), \delta r- \gamma )$.  Let
$\fc$ be the maximal subpath of $\q $ with endpoints contained in $\Bl
(\q (0), \frac{\delta}{2} r- 3\gamma )$.  Then for $\gamma$ large
enough we may assume that $\fc$ is entirely contained in $\Bl (\q (0),
\frac{\delta}{2} r- 2\gamma )$, as $\fc$ is at finite Hausdorff
distance from the geodesic joining its endpoints.  All the points in
$\pgot$ are at distance at least $\frac{\delta}{2} r +\gamma$ from
$\fc \, $.  Let $x_0=a, x_1,...,x_n=b$ be consecutive points on
$\pgot$ dividing it into subarcs of length $\frac{\delta}{2} r$
(except the last who might be shorter).  For each of the points $x_i$
let $x_i'\in \fc$ be a point minimizing the distance to $x_i$.
Lemma~\ref{lem1} implies that $\dist (x_i', x_{i+1}')\leq D_0$ for
every $i$.  For some $\epsilon >0$ and for all $r$ sufficiently large
we have $\dist (x_0' , x_n') \geq \epsilon r$.  Indeed if we would
assume the contrary then there would exist a sequence of positive
numbers $r_n$ diverging to infinity such that $\q(-r_n)$ and $q
(r_n)$ have their respective nearest points $u_n$ and
$v_n$ on $\fc_n$, maximal subpath of $\q$ with endpoints in
$\Bl (\q (0), \frac{\delta}{2} r_n- 3\gamma )$, at distance $\dist (u_n, v_n)$ at most $\frac{1}{n} r_n$.  Then in $\coneomega (X,
\q (0), (r_n))$ we obtain two points on the transversal line $\q_\om$
separated by $B(\q_\om (0), \frac{\delta}{2})$ but whose nearest point
projections onto $\q_\om \cap B(\q_\om (0), \frac{\delta}{2})$
coincide.  This is impossible.

 It follows that $\epsilon r \leq n D_0$, whence $n\geq \frac{\epsilon
 }{D_0}r$.  It follows that the length of $\pgot$ is larger than
 $(n-1) \frac{\delta}{2} r \geq \left( \frac{\epsilon }{D_0}r -1
 \right) \frac{\delta}{2} r$.  In other words the length of $\pgot
 \succeq r^2$.\endproof

\begin{cor}\label{cor3}
Assume that a finitely generated group $G$ acts on a $CAT(0)$--space
$X$ such that one (every) orbit map $G\to X\, ,\, g\mapsto gx$, is a
quasi-isometric embedding and its image contains a Morse quasi-geodesic.  Then
the divergence $\Dv^G \succeq x^2$.
\end{cor}

A phenomenon similar to that in Lemma~\ref{lem1} may occur in
general in a metric space with a quasi-geodesic.

\begin{defn}
In a metric space $X$ a quasi-geodesic $\q$ is called
$D$--\emph{contracting} (or simply \emph{contracting}) if for every
ball $B$ disjoint from $\q$ the points in $\q$ nearest to points in
$B$ compose a set of diameter $D$.
\end{defn}

The following result is immediate.

\begin{lem}
Let $\q$ be a contracting quasi-geodesic.
\begin{enumerate}
  \item The divergence $\Dv^\q \succeq x^2\, $, \item The
  quasi-geodesic $\q$ is Morse.
\end{enumerate}
\end{lem}

The following examples of contracting quasi-geodesics are known:

\begin{enumerate}
  \item cyclic subgroups generated by pseudo-Anosov elements in
  mapping class groups with a word metric \cite{Behrstock:asymptotic};
  \item orbits of fully irreducible elements in
  Outer Space \cite{Algom-Kfir} \cite{Hamenstadt:Linesminimaout};
  \item orbits of pseudo-Anosovs in the Teichm\"uller space with the
  Weil-Petersson metric \cite{Behrstock:asymptotic}.
\end{enumerate}

\begin{cor}\label{cor:out}
Each of the following three metric spaces has quadratic divergence:
\begin{enumerate}
  \item the mapping class group of a surface $S$ with $\xi (S)\geq 2$;
  \item $Out(F_n)$ for $n\geq 3$; \item the Teichm\"uller space with
  the Weil-Peterson metric for surfaces $S$ with complexity $3\times\rm{ genus} + \#\mbox{\rm{(boundary components) }}-3 $ at least $4$, and $S$ is not the one-punctured surface of genus two.
\end{enumerate}
\end{cor}

Note that in \cite{Duchin-Rafi} it is proved that the Teichm\"uller
space with the Teichm\"uller metric has quadratic divergence, and an
alternative proof of the quadratic divergence for the mapping class
group is provided.

To sum up the relationship between divergence and existence of
cut-points in asymptotic cones the following are known:

\me

\begin{itemize}
  \item if all asymptotic cones are without cut-points then the
  divergence is linear;

  \me

  \item if we assume that at least one asymptotic cone is without
  cut-points then examples were constructed by Olshanskii-Osin-Sapir
  of groups $G$ with divergence satisfying $\Dv^G (n) \leq C n$ for a
  constant $C$ for all $n$ in an infinite subset of $\N$ and with
  $\Dv^G\preceq f(n)$ for any $f$ such that $\frac{f(n)}{n}$
  non-decreasing; in particular $\Dv^G$ may be as close to linear as
  possible; but it is superlinear if one asymptotic cone has
  cut-points;

      \me

  \item if the space is $CAT(0)$ and all asymptotic cones have
  cut-points coming from the limit set of a Morse quasi-geodesic then
  the divergence is at least quadratic.
\end{itemize}

This raises the following natural question.

\begin{qn} If a $CAT(0)$ (quasi-homogeneous) metric space has
cut-points in every asymptotic cone, must the divergence of that
metric space be at least quadratic?
\end{qn}

An affirmative answer to this question would be an immediately
corollary of an affirmative answer to the following (which we expect
would be more difficult to establish):

\begin{qn}
Does the existence of cut-points in every asymptotic cone of a
$CAT(0)$ quasi-homogeneous metric space imply the existence of a Morse
quasi-geodesic?
\end{qn}

\section{Morse elements and length of the shortest conjugator}
\label{sec:conjug}

It is known that given a group $G$ acting properly
discontinuously and cocompactly on a $CAT(0)$--space, and two elements
$u,v$ that are conjugate in $G$ there exists $K>0$ depending
on the choice of word metric in $G$ such that $v = gug\iv$ for some
$g$ with $|g|\leq \exp (K (|u|+|v|))$.  As shown below,
a similar estimate on the length of the shortest
conjugator holds in 
a more general context of groups with some
non-positively curved or hyperbolic geometry associated to them.

\subsection{The $CAT(0)$ set-up}

A standard $CAT(0)$ argument yields a bound on the length of the
shortest conjugator
of two axial Morse isometries $u$ and $v$ of a locally compact
Hadamard space in terms of two parameters of the geometry of the
action of both $u$ and $v$.  We define these parameters below.

\begin{enumerate}
  \item Recall that given an axial isometry $u$ of a $CAT(0)$ space
  the set
  $$\mn (u)= \{ x\in X \mid \dist (x,ux ) = \inf_{y\in X} \dist
  (y,uy )\}$$ is isometric to a set $C\times \R\, ,$ where $u$ acts as
  a translation with translation length $t_u$ along each fiber $\{ c\}
  \times \R$ and $C$ is a closed convex subset \cite[Theorem 6.8, p.
  231]{Bridson-Haefliger}.  When the axial isometry is Morse, the set
  $C$ is bounded, and we denote by $D_u$ the diameter of the set
  $C\times [0,t_u]\, $.

      \me

  \item There exists $\theta_u >0$ such that if $x$ is a point outside
  $\nn_1\left( \mn (u) \right)$ and $x'$ is its nearest point
  projection onto $\mn (u)$ then
       \begin{equation}\label{eq:dxux}
	\dist (x,u x)\geq \dist (x', ux') + \theta_u \dist (x,x')\, .
       \end{equation}

       Indeed assume that on the contrary we have points $x_n$ outside
       $\nn_1\left( \mn (u) \right)$ with projections $x_n'$ such that
       $\dist (x_n,u x_n)\leq \dist (x_n', ux_n') + \frac{1}{n} \dist
       (x_n,x_n')$.  By the convexity of the distance we may assume
       that all $x_n$ are at distance $1$ from $\mn (u)$ and by
       eventually applying powers of $u$ we may assume that all $x_n'$
       are in the compact set $C\times [0,t_u]\, $.  Since $X$ is
       locally compact the quadrangles of vertices $x_n, x_n', ux_n',
       ux_n$ converge on a subsequence in the Hausdorff distance to a
       flat quadrangle $a, a', ua', ua$ intersecting $\mn (u)$ in the
       edge $[a',ua']$ and such that for every $z\in [a,a']$, $\dist
       (z,uz) =t_u$.  This contradicts the definition of $\mn (u)$.\end{enumerate}

  \begin{prop}
  Let $X$ be a locally compact Hadamard space, let $G$ be a group of
  isometries of $X$ and let $x_0$ be a point in $X$.

  Let $u$ and $v$ be two Morse axial isometries of $X$ that are
  conjugate in $G$.  Then there exists an element $g\in G$ such that
  $v= g u g\iv$ and such that
  $$
  \dist(x_0, gx_0)\leq \frac{1}{\theta_u } \left[ \dist (x_0, ux_0) +
  \dist (x_0, vx_0)\right] + D_u +2\, .
  $$

  In particular if the map $g\mapsto gx_0$ is a quasi-isometric
  embedding of $G$ into $X$ then $\dist_G (1,g) \leq_{A,B} \dist (1,
  u) + \dist (1, v)\, ,$ where the constants $A,B$ depend on the
  constants of the above isometries, on $\theta_u$ and on $D_u$.
  \end{prop}

  \proof Let $y_0$ be the nearest point projection of $x_0$ onto $\mn
  (u)$ and $z_0$ the nearest point projection of $x_0$ onto $\mn (v)$.

Let $g\in G$ be such that $v= g u g\iv$.  Then $\mn (v) = g\mn (u)$,
in particular $gy_0\in \mn (v)$.  By eventually replacing $g$ with
$v^kg $ for an appropriate $k\in \Z$ one may assume that both $z_0$
and $gy_0$ are in the same isometric copy of $C\times [0,t_u]$, hence within
distance at most $D_u \,$.

The distance $\dist(x_0, gx_0)$ is at most $\dist (x_0, z_0) + D_u +
\dist (y_0, x_0)\, $.  On the other hand either $\dist (x_0, y_0) \leq 1$ or $\dist (x_0, y_0) \leq
\frac{1}{\theta_u } \dist (x_0, ux_0)$. Similarly, either  $\dist (x_0, z_0) \leq 1$ or $\dist (x_0, z_0) \leq
\frac{1}{\theta_u } \dist (x_0, vx_0)$. Hence
$$
\dist(x_0, gx_0)\leq \frac{1}{\theta_u } \left[ \dist (x_0, ux_0) +
\dist (x_0, vx_0) \right] + D_u +2\, .
$$
  \endproof

  When $X$ is a Hadamard manifold the constant $D_u$ equals $t_u$ and
  it is controlled by $\dist (1,u)$ as well.  The problem is to find a
  lower bound for $\theta_u$ uniform for all Morse axial isometries
  $u$.  This is equivalent to the property of having ``uniform
  negative curvature'' in the neighborhood of all the periodic Morse
  geodesics.

\subsection{Conjugators in groups acting acylindrically on trees}

In this subsection we prove the following general result and then give an
application to $3$--manifolds.

\begin{thm}\label{thm:conjacyltree}
     Let $G$ be a group acting cocompactly and $l$-acylindrically on a simplicial tree $T$. For every $R>0$ and for a fixed word metric on $G$ let $f(R)$ denote the supremum of all diameters of intersections $\stab (a)\cap \nn_R( g\stab (b))$, where $a$ and $b$ are vertices in $T$ at distance at least $l$, and $g\in G$ is at distance $\leq R$ from $1$.

There exists a constant $K$ such that if
two loxodromic elements $u, v$ are conjugate in $G$ then there exists $g$
conjugating $u,v$ such that
$$
|g| \leq f(|u|+|v|+K)+|u|+|v|+2K\, .
$$
\end{thm}

\proof  Without loss of generality, by passing to a subtree if necessary, we assume that $G$ acts without inversions of edges. Let $D$ be a finite sub-tree of $T$ (possibly without some endpoints) that is a fundamental domain for the action of $G$ on $T$ \cite[$\S 3.1$]{Serre:trees}. Let $S$ be the fixed finite generating set defining the word metric on $G$. For every $s\in S$ and every vertex $o\in D$ we consider $g_1=1,g_2,...,g_{m-1},g_m=s$ in $G$ such that $o, g_1\cdot o, g_2\cdot o,...,
g_{m-1}\cdot o , s\cdot o $ are the consecutive intersections of
$[o,s\cdot o]$ with $G\cdot o$. We denote by $V(s,o)$ the set of elements $\{ g_1,...,g_m \}$.

 Let $u,v$ be two loxodromic elements in $G$ such that $v=gug\iv$ for some
 $g\in G$.  The element $u$ has a translation axis $A_u$ in $T$.  Likewise $v$ has an
 axis $A_v \subset T$ and $gA_u=A_v$.  Our goal is to control $|g|_S$ in
 terms of $|u|_S + |v|_S$.

 If $D$ intersects $A_u$ then take a vertex $o$ in the intersection. If not, let $p$ be the nearest point to $D$ on $A_u$ and consider the unique vertex $o\in D$ and an element $h\in G$ such that $p=h0\, $.

For each $g\in G$ we write $\pi(g)$ to denote
$g\cdot o$. Also for every geodesic $[a,b]$ in the Cayley graph of $G$, with consecutive vertices $g_1=a,g_2,...,g_m=b$ we denote by $\pi [a,b]$ the path in the tree $T$ composed by concatenation of the consecutive geodesics $[g_1o, g_2o],...[g_{m-1}o, g_m o]$.

 \begin{lem}\label{lem9}
For every $R\ge 0$ and every pair of vertices $a,b$ in the orbit $G\cdot o$ with $\dist(a,b)\geq l$
    the set $V_{a,b}$ of elements $g\in\pi\iv(a)$ such
that $\dist(g,\pi\iv(b))\le R$, if non-empty, has diameter at most $f(R)$.
\end{lem}

\proof Without loss of generality we may assume that $a= o$, hence $\pi\iv(a) = \stab (o)$, and that there exists $g$ in $\stab (o)$ at distance at most $R$ from $\pi\iv(b)$. Without loss of generality we may assume that $g=1$. The set $\pi\iv(b)$ can then be written as $h\stab (o)$, where $h\in B(1,R)$, $h \cdot o = b\, $.

Every $g\in V_{o,b}$ is then in $\stab (o)$ and in $\nn_{R} (h\stab (b))$. It follows from the hypothesis that the diameter of $V_{o,b}$ is at most $f(R)\, $.\endproof

 Consider the geodesic $[1, u^m]$, for some fixed large enough power $m$.  Its image by $\pi$ covers the geodesic $[o, u^m o]$. The latter geodesic contains the two points $ho$ and $u^mho$.  Then $[go,
 gu^mo]$ intersects $A_v$ in $[gho, gu^mho]$, which can also be written as $[ko, v^mko]$ for $k=gh$.

 Likewise, the image under $\pi$ of the geodesic $[1, v^m]$ contains the geodesic $[o, v^mo]$ in $T$,
 and the latter geodesic contains a sub-geodesic of $A_v$ of the form
 $[ro, v^{m-1}ro]$, with $r=v^i ko$ for some $i\in \Z$.
 By possibly post-composing $g$ with $v^{i}$ we may assume
 that $ko$ and $ro$ coincide. If $m\geq l+1$ then $[ro, v^{m-1}ro]$ is of length at least
 $l$.  This implies that the geodesic $[1, v]$ in the Cayley graph
 contains two pairs of consecutive vertices $v_1, v_1'$ and $v_2,v_2'$
 such that $ro \in [v_1o , v_1'o]$ and $v^{m-1}ro \in [v_2 o, v_2'o]$.

 Likewise $[1,u]$ contains
 two pairs of consecutive vertices $u_1, u_1'$ and $u_2,u_2'$ such that $ho \in [u_1o , u_1'o]$ and $u^{m-1}ho \in [u_2 o, u_2'o]$.  We thus have that $ro = v_1x_1o = gu_1x_1'o$, where $x_1\in V(v_1\iv v_1',o)$ and $x_1'\in V(u_1\iv u_1',o)$; and that $v^{m-1}ro = v_2 x_2 o = gu_2 x_2'o$, where $x_2 \in V(v_2\iv v_2',o)$ and $x_2'\in V(u_2\iv u_2',o)$.

 We denote by $M$ the maximum of all the $|x|_S$ for all $x$ in all sets of the form $V(o,s)$ for a vertex $o\in D$ and $s\in S$. The above implies that both $v_1x_1$ and $gu_1x_1'$ are in
 $\pi\iv (ro)$ and at distance at most $|u|+|v|+M$ from $\pi\iv (v^{m-1}ro)$.
 Lemma \ref{lem9} implies that $v_1$ and $gu_1$ are within distance at
 most $f(|u|+|v|+M) +2M$.  It follows that
 $$
 |g| \leq |v| + f(|u|+|v|+M) +|u|+2M\, .
 $$
\endproof

The above general results imply in particular a linear control of the
shortest conjugator for Morse geodesics in (non-geometric) $3$--manifolds.

\begin{cor}\label{thm:conjgraphmfld}
Let $M$ be a non-geometric prime 3-dimensional manifold and let $G$ be its
fundamental group.

For every word metric on $G$ there exists a constant $K$ such that if
two Morse elements $u, v$ are conjugate in $G$ then there exist $g$
conjugating $u,v$ such that
$$
|g| \leq K(|u|+|v|)\, .
$$
\end{cor}

\proof Since $M$ is non-geometric, it can be cut along tori and Klein
bottles into finitely many geometric components that are either
Seifert or hyperbolic.  We will apply
Theorem~\ref{thm:conjacyltree} by considering the Bass-Serre tree $T$
associated to the geometric splitting of $M$ described before and
recalling that two elements in $\pi_1 (M)$ are Morse if and only if
they are either loxodromic elements for the action on $T$ or both
contained in a hyperbolic component of $M$/$\pi_1 (M)$.

The following proposition of Kapovich--Leeb, combined with the fact that for every non-geometric prime 3-dimensional manifold $M$ there exists a non-positively curved such manifold $N$, and a biLipschitz homeomorphism between the universal covers $\widetilde{M}$ and $\widetilde{N}$ preserving the components \cite[Theorem 1.1]{KapovichLeeb:3manifolds} implies that for $l\geq 3$ the function $f(R)$
 defined in Theorem~\ref{thm:conjacyltree} is of the form $2R + \kappa$.

\begin{prop}[\cite{KapovichLeeb:Haken}]

Let $M$ be a non-geometric prime 3-dimensional manifold admitting a non-positively curved Riemannian metric. There exists a constant
$\kappa >0$ dependent only on $M$ such that given two
geometric components $C,C'$ of $\widetilde{M}$ separated by two flats, the nearest point projection of $C'$ onto $C$ has diameter at
most $\kappa$.
\end{prop}

This settles the case when both $u$ and $v$ are loxodromic elements.

Assume now that $u$ and $v$ both stabilize hyperbolic components.
Assume that we have fixed a basepoint $x_0$ in the universal cover
$\widetilde{M}$.  The map $g\mapsto gx_0$ is a quasi-isometry with
fixed constants depending only on the given word metric on~$G$.
Assume that $u$ stabilizes the hyperbolic component $H\subset
\widetilde{M}$ and that it acts on this component as a loxodromic
element.  We see $H$ as a subset of $\mathbb{H}^3$.  Let $A_u$ be the
geodesic axis in $\mathbb{H}^3$ on which $u$ acts by translation,
denote the translation length by~$t$.  Note that every segment of
length $t$ on $A_u$ intersects $H$.  We denote by $\mathrm{Sat} (A_u)$
the set obtained from $A_u$ by replacing its intersections with the
open horoballs that compose $\mathbb{H}^3 \setminus H$, with the
corresponding boundary horospheres.

The element $v=gug\iv$ stabilizes a hyperbolic component
$H'= gH\subset \widetilde{M}$, and there exists a geodesic axis $A_v =
gA_u$ in $\mathbb{H}^3  \supset H'$ such that $v$ acts on this axis by translation with
translation length~$t$. We define $\mathrm{Sat} (A_v)$ similarly.

Let $x_0'$ be the nearest point projection of $x_0$ on $H$, let $y_0'$
be the nearest point projection of $x_0'$ onto $A_u$ and let $y_0\in
\mathrm{Sat} (A_u)$ be either the intersection point of $[x_0', y_0']$
with a boundary horosphere if $y_0'$ is in $\mathbb{H}^3 \setminus H$,
or equal to $y_0'$ if this latter point is in $H$.  Note that $ux_0'$
will be on a different boundary horosphere than $x_0'$, and the same
for $uy_0$ and $y_0$, if $y_0$ is on a boundary horosphere.  According
to \cite[Lemma 4.26]{DrutuSapir:TreeGraded}, the geodesic
$\pg_{y_0,uy_0}$ joining $y_0$ to $uy_0$ is contained in a
$\delta$--neighborhood of $\mathrm{Sat} (A_u)$, moreover due to the
fact that every segment of length $t$ on $A_u$ intersects $H$, it
follows that $\pg_{y_0,uy_0}$ intersects the $\delta$--neighborhood of
$A_u$.

Due to the fact that the metric space
$\widetilde{M}$ is hyperbolic relative to the connected components of $\widetilde{M}\setminus \mathrm{Interior} (H)$, it follows that the concatenation of the geodesics $[x_0, x_0'], [x_0', y_0], \pg_{y_0,uy_0}, [uy_0, ux_0'], [ux_0', ux_0]$ composes an $(L,C)$--quasi-geodesic, with $L\geq 1$ and $C\geq 0$ depending only on $M$ \cite[Lemma 8.12]{DrutuSapir:TreeGraded}. We denote this quasi-geodesic $\q_{x_0 , ux_0}$. We construct in a similar manner an $(L,C)$--quasi-geodesic $\q_{x_0 , vx_0}$ joining $x_0$ and $vx_0$ and containing in its
$\delta$--neighborhood a sub-segment of the axis $A_v$. Note that the $(L,C)$--quasi-geodesic $g\q_{x_0, ux_0}$ joining $gx_0, gux_0$ contains in its $\delta$--neighborhood another sub-segment of $A_v$.  By pre-composing $g$ with a
power of $v$ and possibly replacing $u,v$ by large enough powers, we may assume that the two sub-segments of $A_v$ mentioned above are the same. In particular the $\delta$-neighborhoods
of $\q_{x_0 , vx_0}$ and of $g\q_{x_0, ux_0}$ intersect.  It follows that
$$
\dist (x_0, gx_0) \preceq \dist (x_0, vx_0) + \dist (x_0, ux_0)\, .
$$
\endproof

\begin{cor}\label{thm:conj3man}
Let $M$ be a $3$-dimensional prime manifold, and let $G$ be its
fundamental group.

For every word metric on $G$ there exists a constant $K$ such that if
two elements $u, v$ are conjugate in $G$ then there exist $g$
conjugating $u,v$ such that
$$
|g| \leq K(|u|+|v|)^2\, .
$$
\end{cor}

\proof Assume first that $M$ is non-geometric, hence decomposable by
tori and Klein bottles into hyperbolic and Seifert components.  The
only case not covered by Corollary~\ref{thm:conjgraphmfld}
is when both $u$ and $v$ stabilize a
Seifert component, i.e., are contained in two groups which are
virtually $\Z \times F_n$. In this case one can easily find a
conjugator of quadratic length.

When $M$ is a geometric
nilmanifold, i.e., when $\pi_1 (M)$ is $2$--step
nilpotent the quadratic upper bound for a conjugator length is proved
in \cite[Proposition~2.1.1]{JOR-conj}.

When $M$ is a geometric
solmanifold, the linear upper bound for a conjugator length is proved
in \cite{Sale:thesis}.

The other geometric cases are
easy. \endproof

\subsection{Conjugators in mapping class groups}
In what follows $S$ denotes a compact oriented surface of genus $g$ and with $p$ boundary components and $\xi (S) = 3g+p-3$ denotes the complexity of the surface.

We prove linear control of the shortest
 conjugator of infinite order elements in the mapping class group by
 providing a new proof of the following result which was established by
 Masur--Minsky \cite[Theorem~7.2]{MasurMinsky:complex2} in the pseudo-Anosov case and by J.~Tao
 \cite[Theorem~B]{JTao:conj} in the reducible case.

 \begin{thm}\label{thm:shortestgen}
There exists a constant $C$ depending only on the surface $S$ and the
fixed generating set of $\MCG (S)$ such that for every two conjugate
elements of infinite order $u$ and $v$ there exists $g$ such that
$v=gug\iv$ and
$$
|g| \leq C [|u| + |v|]\, .
$$
 \end{thm}

 It is worth noting that the mapping class group is not $CAT(0)$,
 c.f., \cite{KapovichLeeb:actions} or \cite{Bridson-Haefliger}.
 Nevertheless, there exists a natural analogue of the
 inequality (\ref{eq:dxux}) from the $CAT(0)$ setting which holds
 here; this will be explained further in the proof below.

\subsubsection*{Background}

We will use a quasi-isometric model of a mapping class group, the
\emph{marking complex}, $\MM(S)$, defined as follows.  Its vertices,
called \emph{markings}, consist of the following pair of data:
\begin{itemize}
    \item \emph{base curves}: a multicurve consisting of $\xi (S)$
    components, i.e. a maximal simplex in $\CC(S)$.  This collection
    is denoted $\base(\mu)$.

    \item \emph{transversal curves}: to each curve
    $\gamma\in\base(\mu)$ is associated an essential curve.  Letting
    $T$ denote the complexity $1$ component of $S\setminus
    \bigcup_{\alpha\in \base(\mu ), \alpha \neq \gamma}\alpha$, the
    transversal curve to $\gamma$ is a curve $t(\gamma)\in\CC(T)$ with
    $\dist_{\CC(T)}(\gamma,t(\gamma))=1$.
\end{itemize}

Two vertices $\mu,\nu$ in the marking complex $\MM(S)$ are connected by an edge if either of
the two conditions hold:

\begin{enumerate}
    \item \emph{Twists}: $\mu$ and $\nu$ differ by a Dehn twist along
    one of the base curves: $\base(\mu)=\base(\nu)$ and all their
    transversal curves agree except for $t_{\mu}(\gamma)$, obtained
    from $t_{\nu}(\gamma)$ by twisting once about the curve $\gamma$.

    \me

    \item \emph{Flips}: The base curves and transversal curves of
    $\mu$ and $\nu$ agree except for one pair $(\gamma,
    t(\gamma))\in\mu$ for which the corresponding pair in $\nu$ consists of the
    same pair but with the roles of base and transversal reversed.
\end{enumerate}

  Note that after performing one flip the new base curve may now
  intersect several other transversal curves.  Nevertheless by \cite[Lemma
  2.4]{MasurMinsky:complex2}, there is a finite set of natural ways to
  resolve this issue which, in turn, yields a uniformly bounded on the diameter of possible markings which can be obtained by flipping the pair $(\gamma,
  t(\gamma))\in\mu$; an edge connects each of these possible flips to
  $\mu$.

\begin{thm}[\cite{MasurMinsky:complex2}]\label{MM:marking} The graph
$\MM(S)$ is locally finite and the mapping class group acts
cocompactly and properly discontinuously on it.  In particular, the
orbit map yields a quasi-isometry from $\MCG(S)$ to $\MM(S)$.
\end{thm}

Given a simplex $\Delta$ in the curve complex $\CC(S)$, we define
$\QQ(\Delta)$ to be the set of elements of $\MM(S)$ whose bases
contain $\Delta$. We recall that there is a coarsely defined
closest point projection map from $\MM(S)$ to $\QQ(\Delta)$ which is
coarsely Lipschitz.

\subsubsection*{Proof of Theorem~\ref{thm:shortestgen}}

We assume that $S$ is a surface with $\xi (S) >1$, otherwise the mapping class
  group is hyperbolic and the result is standard. We make use of two cocompact actions of  $\MCG(S)$: the above mentioned properly discontinuous action on the marking complex $\MS$ and an action far from properly discontinuous on the
 curve complex $\CCS$. Neither of the two complexes $\CCS$ nor $\MS$
 are $CAT(0)$.

We begin with the case of two conjugate pseudo-Anosov elements. Our goal is to find a natural analogue of the
 inequality (\ref{eq:dxux}) from the $CAT(0)$ setting. The difficulty is that a pseudo-Anosov element is loxodromic with a translation axis in $\CCS$, which makes it hard to find an appropriate definition of a projection of an element in $\MS$ to it.

 Let $k$ be a
  pseudo-Anosov.  According to \cite[Theorem 1.4]{Bowditch:tightgeod},
  there exists $m=m(S)$ such that $k^m$ preserves a bi-infinite
  geodesic $\pg_{k}$ in $C(S)$.  For every curve $\gamma$ denote by $
  \gamma'$ a closest point to it on $\pg_{k}$. Let $\widehat{k} = k^m\, $. A standard hyperbolic
  geometry argument implies that for every $i\geq 1$
\begin{equation}\label{eq:transl}
\dcs (\gamma , \widehat{k}^{i} \gamma )\geq \dcs (\gamma' , \widehat{k}^{i} \gamma' ) +
O(1)\geq i+O(1)\, .
\end{equation}

Let $\mu$ be an arbitrary element in $\MS$ and let $\gamma$ be a closest point to $\pcs (\mu)$ on $\pg_{k}$.  A hierarchy path $\fh$ joining
$\mu$ and $\widehat{k}\mu$ contains two points $\nu , \nu'$ such that:
\begin{itemize}
  \item the subpath with endpoints $\mu , \nu$ is at $\CCS$--distance
  $O(1)$ from any $\CCS$--geodesic joining $\pcs (\mu)$ and $\gamma$;

  \me

  \item the subpath with endpoints $\widehat{k}\mu , \nu'$ is at $\CCS$--distance
  $O(1)$ from any $\CCS$--geodesic joining $\pcs (\widehat{k}\mu)$ and $\widehat{k}\gamma$;

  \me

  \item if the translation length of $\widehat{k}$ along $\pg_{k}$ is large enough then the subpath with
  endpoints $\nu , \nu'$ is at $\CCS$-distance $O(1)$ from $\pg_{k}$;

  \me

  \item $\dcs (\nu' , \widehat{k}\nu) $ is $O(1)$.
\end{itemize}

Note that by equation~(\ref{eq:transl}) there exists an integer $N=N(S)$ such
that for every pseudo-Anosov $k$ the power $k^N$ preserves a
bi-infinite geodesic $\pg_{k}$ in $\CCS$, and every subpath with endpoints
$\nu , \nu'$ defined as above is at $\CCS$-distance $O(1)$ from
$\pg_{k}$.

The group $\MCG(S)$ acts
co-compactly on $\MM(S)$, therefore there exists a compact subset $K$ of $\MM
(S)$ such that $\MCG(S) K= \MM (S)$. We pick a basepoint $\mu_0$ in $K$.
The map $\MCG (S) \to \MM (S)\, ,\, g \mapsto g\mu_0$ is a
quasi-isometry, by Theorem~\ref{MM:marking}.

Let $u$ and $v$ be an arbitrary pair of pseudo-Anosovs for which there
exists $g\in \MCG(S)$ such that $v= gug\iv$. Our goal is to prove that for an appropriate choice of $g$,
 $\dist (\mu_0 , g\mu_0)$ is controlled by a linear function of $\dist (\mu_0, u\mu_0 ) + \dist (\mu_0, v\mu_0 )$.

If $\pg_u$ and $\pg_v$
are axes in $\CC(S)$ defined as above then $\pg_v= g\pg_u\, $.  Up to
replacing $u$ and $v$ by their $N$-th powers, we may assume that both
preserve their respective axes $\pg_u$ and $\pg_v$, and every subpath
with endpoints $\nu , \nu'$ defined as above is at $\CC(S)$--distance
$O(1)$ from $\pg_u$, respectively $\pg_v$.

Let $\fh$ be a hierarchy path joining $\mu_0$ and $u\mu_0$ and let
$\nu$ and $\nu'$ be two points on it defined as above.  Then $g\fh$ is
a hierarchy path joining $g\mu_0$ and $gu\mu_0= vg\mu_0$ and $g\nu$,
$g\nu'$ satisfy similar properties for the path $g\fh$, the
pseudo-Anosov $v$ and its axis $\pg_v= g\pg_u\, $.

Now let $\fk$ be a hierarchy path joining $\mu_0$ and $v\mu_0$ and let
$\xi$ and $\xi'$ be the two points on it defined as above.

By eventually replacing $g$ with $v^kg$, for an appropriate $k\in \Z$,
we may assume that $g\nu$ and $\xi$ are at $\CC(S)$ distance at most
$t_v+ O(1)\, ,$ where $t_v$ is the translation length of $v$ along
$\pg_v\, $. There are two
cases to discuss.  In order to define the necessary parameters, recall
the following result.

\begin{thm}[Masur--Minsky;
\cite{MasurMinsky:complex2}]\label{distanceformula} If $\mu,\nu\in
\MM(S)$, then there exists a constant $K(S)$, depending only on $S$,
such that for each $K>K(S)$ there exists $a\geq 1$ and $b\geq 0$ for
which:
    \begin{equation}\label{fdistformula}
     \dist_{\MM(S)}(\mu,\nu) \approx_{a,b} \sum_{Y\subseteq S} \Tsh
     K{\dist_{\CC(Y)}(\pi_{Y}(\mu),\pi_{Y}(\nu))}.
\end{equation}
\end{thm}

In particular this implies that there exists $\kappa>0$ and $A,B$
depending only on $S$ such that if $\dist_{\MM(S)} (\mu, \nu) \geq
\kappa \dcs (\mu, \nu)\, $ then
    \begin{equation}\label{fdistformulaprop}
     \dist_{\MM(S)}(\mu,\nu) \approx_{A,B} \sum_{Y\subsetneq S} \Tsh
     K{\dist_{\CC(Y)}(\pi_{Y}(\mu),\pi_{Y}(\nu))}.
\end{equation}

\me

In the formulas above we use the following notation. For two numbers $d\geq 0$ and $K\geq 0$, $\Tsh K {d}$ is equal to $d$ if $d\geq K$, and it is zero otherwise.

\me

The subsurfaces that appear in \eqref{fdistformulaprop} for a given pair $\mu , \nu$ and a given constant $K>K(S)$ are called $K$--\emph{large domains} of that pair. The proper subsurfaces are called $K$--\emph{large proper domains}. We omit $K$ when irrelevant.

\me

\noindent \emph{Case 1.} Assume that $\dist_{\MM(S)} (\mu_0, g\mu_0) \leq
\kappa \dcs (\mu_0, g\mu_0)\, $.

Note that $\dcs (\mu_0, g\mu_0) \leq \dcs (\mu_0, \xi )+ \dcs (\mu_0,
\nu ) + t_v+ O(1)\leq 2 \dcs (\mu_0 , v\mu_0 ) + \dcs (\mu_0, u
\mu_0)+ O(1)$, which implies that $\dist_{\MM(S)} (\mu_0,
g\mu_0)\leq_{A',B'}\dist (\mu_0, u\mu_0) + (\mu_0, v\mu_0)\, $, where $A', B'$ depend on $\kappa, a,b$ and $K$ from Theorem \ref{distanceformula}.

\me

\noindent \emph{Case 2.} Assume that $\dist_{\MM(S)} (\mu_0, g\mu_0) \geq
\kappa \dcs (\mu_0, g\mu_0)\, $.

This together with equation~(\ref{fdistformulaprop}) then implies:
\begin{equation}\label{eq:domains}
 \dist_{\MM(S)} (\mu_0, g\mu_0) \approx_{A,B} \sum_{Y\subsetneq S} \Tsh
 K{\dist_{\CC(Y)}(\pi_{Y}(\mu_0),\pi_{Y}(g\mu_0))}\, .
\end{equation}

    Recall that the point nearest to $\pcs (\mu_0)$ on the axis
    $\fg_v$ is at $\CC (S)$--distance $O(1)$ from $\xi$ (actually,
    this may not be a point, but since $\CC(S)$ is hyperbolic the
    set of closest points form a
    bounded diameter; hence we abuse notation slightly, as this set is
    coarsely a point).  Likewise the point nearest to $\pcs (g\mu_0)$
    on $\fg_v$ is at $\CC (S)$--distance $O(1)$ from $g\nu$ and at
    distance $t_v + O(1)$ from $\xi$.  Every proper subsurface $Y$
    appearing in the sum (\ref{eq:domains}) has the property that
    every hierarchy path joining $\mu_0$ to $g\mu_0$ intersects
    $Q(\partial Y)$.  Hence $\partial Y$ is at $\CC (S)$--distance
    $O(1)$ from the union of two geodesics in $\CC (S)$ joining $\pcs
    (\mu_0)$ respectively $\pcs (g\mu_0)$ to their nearest points on
    $\fg_v\, $ with the arc of $\fg_v$ with endpoints these two
    nearest points.  It follows that a nearest point to $\pcs
    (\partial Y)$ on the axis $\fg_v$ is at $\CC (S)$--distance at
    most $t_v + O(1)$ from $\xi$.

    The analogue of equation~(\ref{eq:domains}) is also satisfied by $v\mu_0$
    and $vg\mu_0 = gu\mu_0\, $.  In particular for every proper
    subsurface $Y'$ appearing in that formula, a nearest point to $\pcs (\partial Y')$ on the axis $\fg_v$ is at $\CC (S)$--distance at most $t_v + O(1)$ from $v\xi$.  Then by replacing both $u$ and $v$ with
    their $J$-th power, for some $J=J(S)$ and arguing with the
    corresponding hierarchy paths $\fh$ joining $\mu_0$, $u^J(\mu_0)$,
    respectively $\mu_0$, $v^J(\mu_0)$, we may assume that the pairs
    $\mu_0, g\mu_0$ and respectively $v\mu_0$, $vg\mu_0$ have no large proper domain in common.

    It follows that for every large proper domain $Y$ of the pair
    $\mu_0, g\mu_0$,
    $$\dist_{\CC(Y)}(\pi_{Y}(\mu_0),\pi_{Y}(g\mu_0))\leq
    \dist_{\CC(Y)}(\pi_{Y}(\mu_0),\pi_{Y}(v\mu_0))+
    $$
    $$
    \dist_{\CC(Y)}(\pi_{Y}(g\mu_0),\pi_{Y}(vg\mu_0))+K\, .$$

    In particular for $K>K(S)$, where $K(S)$ is the constant from Theorem \ref{distanceformula}, if we consider $Y$ a $3K$--large proper domain for the pair $\mu_0, g\mu_0$, it must be a $K$--large proper domain either for $\mu_0, v\mu_0$ or for $g\mu_0 , vg\mu_0$ or for both pairs. We may then write that
    $$
\sum_{Y\subsetneq S} \Tsh
     {3K}{\dist_{\CC(Y)}(\pi_{Y}(\mu_0),\pi_{Y}(g\mu_0))}\leq 3 \sum_{Y\subsetneq S} \Tsh
     K{\dist_{\CC(Y)}(\pi_{Y}(\mu_0),\pi_{Y}(v\mu_0))} +
     $$
     $$3 \sum_{Y\subsetneq S} \Tsh
     K{\dist_{\CC(Y)}(\pi_{Y}(g\mu_0),\pi_{Y}(gu\mu_0))}
    $$
     whence
    $$
  \dist_{\MM(S)} (\mu_0, g\mu_0) \preceq_{A',B'} \dist_{\MM(S)} (\mu_0, v\mu_0) + \dist_{\MM(S)}
  (\mu_0, u\mu_0) \, .
    $$

    This completes the proof of Theorem~\ref{thm:shortestgen} for pairs of pseudo-Anosov elements.  We now proceed to the full proof of Theorem~\ref{thm:shortestgen}.

\proof Let $u,v\in\MM$.  By replacing $u,v$ with sufficiently high
powers $u^N, v^N$, where $N=N(S)$, we may assume that the two elements are pure.

\begin{lem}\label{lem:betw}
Let $\nu$ and $\rho$ be two points in $\MM(S)$, let $\Delta$ be a
multicurve, and let $\nu' , \rho' $ be respective nearest point
projections of $\nu , \rho $ on $\QQ(\Delta)$.  Assume there exist
$U_1,...,U_k$ subsurfaces such that $\Delta = \partial U_1 \cup
...\cup \partial U_k$, and $\dist_{C(U_i)} (\nu , \rho ) >M$ for every
$i=1,...,k$, where $M=M(S)$ is a large enough constant.

Then for every $\fh_1$, $\fh_2$ and $\fh_3$ hierarchy paths joining
$\nu , \nu'$ respectively $\nu' , \rho' $ and $\rho' , \rho$, the path
$\fh_1\sqcup \fh_2 \sqcup \fh_3$ has length $\approx_{a,b} \dist_{\MM
(S)} (\nu , \rho )$, where $a,b$ depend only on the topological type
of $\Delta$.
\end{lem}

\proof This follows by a limiting argument from
\cite[Lemma 4.27]{BehrstockDrutuSapir:mcg} and
\cite[Theorem 4.16]{BehrstockDrutuSapir:mcg}.\endproof

Let $\Delta_u$ be a multicurve such that if $U^1,...,U^m$ are the
connected components of $S\setminus \Delta_u$ and the annuli with core
curve in $\Delta_u$ then $u$ is a pseudo-Anosov on $U^1,...,U^k$ (Dehn
twists are assumed to be pseudo-Anosovs on annuli) and the identity
map on $U^{k+1},..., U^m$, and $\Delta_u= \partial U^1\cup...\cup
\partial U^k$ (the latter condition may be achieved by deleting the
boundary between two components on which $u$ acts as identity).
Similarly, for $v$ we consider the multicurve $\Delta_v$ and
$V^1,...,V^m$.

 Then $g\Delta_u=\Delta_v$ and $gU_i = V_i$, up to reordering $V^1,...,V^m$.

 Let $\nu$ and $\xi$ be nearest point projections of $\mu_0$ onto
 $\QQ(\Delta_u)$ and respectively $\QQ(\Delta_v)$. By
 eventually replacing $u,v$ with large enough powers we may assume that Lemma
 \ref{lem:betw} applies to the pairs $\mu_0, u\mu_0$ and $\mu_0,
 v\mu_0$, respectively.

Like for pseudo-Anosovs, we have two cases.

\me

\noindent {\bf Case 1.} Assume that
$$\dist_{\MM(S)}(\mu_0 , g\mu_0) \approx_{A,B} \sum_{Y\subseteq S,
Y\pitchfork \Delta_v}
    \Tsh K{\dist_{\CC(Y)}(\pi_{Y}(\mu_0),\pi_{Y}(g\mu_0 ))}\, .
    $$

    Then the same relation is true for $v\mu_0 , vg\mu_0$.  By
    eventually replacing $v$ by a power of itself we may assume that the
    set of proper domains appearing in the sum above is disjoint from the
    corresponding set of proper domains for $v\mu_0 , vg\mu_0$.  It follows
    that all the large proper domains for the pair $\mu_0 , g\mu_0$ are large proper domains
    either for $\mu_0, v\mu_0$ or for $g\mu_0, vg\mu_0= gu\mu_0\, $.

    We discuss the case of the whole surface $S$ separately. Consider a tight geodesic $\pg_u$ in $\CC (S)$ joining $\pcs (\mu_0 )$ to $\Delta_u\, $. We state that by replacing $u$ with a large enough power we may ensure that points on $u\pg_u$ at $\delta$-distance from $\pg_u$ are at distance at most $D$ from $\Delta_u$, for some $D>0$, where $\delta>0$ is the hyperbolicity constant of $\CC (S)\, $. Indeed assume that there exists a point $a$ on $\pg_u$ at distance at least $D$ from $\Delta_u$ such that $B(a,\delta )$ intersects $u^i \pg_u$ for $i\in \{1,2,...,N\}$. This in particular implies that $\dist_{\CC (S)} (a, u^i a) \leq 2\delta $.

    Let $b$ be the point on $\pg_u$ at distance $\frac{D}{2}$
    from $\Delta_u$.  By the above, every $u^i b$ is in $B(b,
    2\delta )$.  On the other hand, by \cite[Theorems 1.1 and
    1.2]{Bowditch:tightgeod} if $D\geq D_0 (S, \delta )$ then there
    exists $m=m(S,\delta )$ such that $B(b, 2\delta )$ contains at
    most $m$ points from the union of tight geodesics $\pg_u \cup
    \bigcup_{i=1}^N u^i \pg_u\, $.  It follows in particular that if
    $N=m+1$ then there exist $i< j$ with $i,j\in \{1,2,...,N\}$ such
    that $u^i b = u^j b$, hence $u^{j-i} b=b \, $.  But,
    for $D$ large enough $b$ together with any curve from $\Delta_u$
    fills the surface, hence it cannot be fixed by a power of $u$.  We
    obtained a contradiction.  Thus we conclude that there exists
    $k\leq N = N(S, \delta )$ such that the intersection of the
    $\delta$-neighborhood of $\pg_u$ with the $\delta$-neighborhood of
    $u^k \pg_u$ is contained in the $D$--neighborhood of $\Delta_u$.

    Likewise we argue that given a tight geodesic $\pg_v$ in $\CC (S)$ joining $\pcs (\mu_0 )$ to $\Delta_v\, $ there exists $r\leq N$ such that the intersection of the $\delta$-neighborhood of $\pg_v$ with the $\delta$-neighborhood of $v^r \pg_v$ is contained in the $D$--neighborhood of $\Delta_v$. We deduce that $\dcs (\mu_0 , \Delta_u )\leq \dcs (\mu_0 , u^k \mu_0 ) + O(1) \leq k \dcs (\mu_0 , u \mu_0 ) + O(1)$ and that $\dcs (\mu_0 , \Delta_v )\leq r \dcs (\mu_0 , v \mu_0 ) + O(1)$. Then $\dcs (\mu_0 , g\mu_0 ) \leq \dcs (\mu_0 , \Delta_v )+ \dcs (\Delta_v , g\mu_0 )\leq N[\dcs (\mu_0 , v \mu_0 ) + \dcs (\mu_0 , u \mu_0 ) ] + O(1) $.

\me

\noindent {\bf Case 2.} Assume that
$$\dist_{\MM(S)}(\mu_0 , g\mu_0) \approx_{A,B} \sum_{Y\subseteq S,
Y\notpitchfork \Delta_v}
    \Tsh K{\dist_{\CC(Y)}(\pi_{Y}(\mu_0),\pi_{Y}(g\mu_0))}\, .
    $$ In other words $\dist_{\MM(S)}(\mu_0 , g\mu_0) \approx_{A',B'}
    \dist_{\MM(S)}(\xi , g\nu)$. Note that it suffices to bound
    $\dist_{\MM(S)}(\xi , g\nu)$ by a multiple of $\dist_{\MM(S)}(\xi
    , v\xi)+ \dist_{\MM(S)}(g\nu , vg\nu)$.  Up to replacing $g$ by
    $zg$ for some element in the centralizer of $v$ we may assume that
    $\xi$ and $g\nu$ have projections on $\MM (U^j)\, ,\, k+1 \leq
    j\leq m \, ,$ at bounded distance.  Thus any large domain $Y$ for $\xi
    , g\nu$ must satisfy $Y\subseteq U^j$ for some $j$ in $\{
    1,2,...,k \}$.

For every $j\in \{1,2,...,k\}$ recall that $v$ restricted to
$V^j$ coincides with a pseudo-Anosov $v_j$.  We use the same notation
$v_j$ to denote the mapping class that acts as $v_j$ on $V^j$ and as
identity on $S\setminus V^j$.  A hierarchy path joining $\xi$ to
$g\nu$ projects onto a quasi-geodesic $\q_j$ in $\CC (V^j)$ containing
in a tubular neighborhood of radius $O(1)$ all the multicurves
$\partial Y$ where $Y\subsetneq V^j$ is a large domain for $\xi$ and
$g\nu$.  By eventually pre-composing $g$ with a power of $v_j$ (hence
with an element in the centralizer of $v$), we may assume that the sub-arc
of $\q_j$ contained in a $O(1)$--tubular neighborhood of the
translation axis of $v_j$ has length $\preceq t_{v_j}$, where
$t_{v_j}$ is the translation length of $v_j$.  Then, by eventually
replacing $v$ with a large enough power, we may assume that the set of
large domains for $\xi , g\nu$ that are proper sub-surfaces of $V^j$
has nothing in common with the set of large domains for $v\xi , vg\nu$
that are proper sub-surfaces of $V^j$. Hence they are all large
domains either for $\xi , v\xi$ or for $g\nu , vg\nu$.

    If $U^j$ is a large domain itself for $\xi , g\nu$ then by arguing as in
    the pseudo-Anosov case (and noting that the copy of $\Z^k$
    generated by the pseudo-Anosov components $v_j$ is in the
    centralizer
    of $v$) we may prove that, by eventually post-composing $g$ with
    an element in the centralizer of $v$, we may ensure that
    $\dist_{C(U^j)} (\xi , g\nu )\preceq \dist_{C(U^j)} (\xi , v\xi )
    + \dist_{C(U^j)} (g\nu , vg\nu )$.\endproof

\def\cprime{$'$}
\providecommand{\bysame}{\leavevmode\hbox to3em{\hrulefill}\thinspace}
\providecommand{\MR}{\relax\ifhmode\unskip\space\fi MR }
\providecommand{\MRhref}[2]{%
  \href{http://www.ams.org/mathscinet-getitem?mr=#1}{#2}
}
\providecommand{\href}[2]{#2}


\begin{thebibliography}{{Ham}09}

\bibitem[AK]{Algom-Kfir}
Y.~Algom-Kfir, \emph{{Strongly Contracting Geodesics in Outer Space}}, preprint
  {\textsc{arXiv:0812.1555}}.

\bibitem[Bal85]{Ballmann:rank}
Werner Ballmann, \emph{Nonpositively curved manifolds of higher rank}, Ann. of
  Math. (2) \textbf{122} (1985), no.~3, 597--609.

\bibitem[Bal95]{Ballmann:book}
\bysame, \emph{Lectures on spaces of nonpositive curvature}, DMV Seminar,
  vol.~25, Birkh\"auser Verlag, Basel, 1995, With an appendix by Misha Brin.

\bibitem[BC]{BehrstockCharney}
J.~Behrstock and R.~Charney, \emph{Divergence and quasi-morphisms of
  right-angled Artin groups}, preprint \textsc{arXiv:1001.3587}, 2010.

\bibitem[BDM09]{BehrstockDrutuMosher:thick1}
J.~Behrstock, C.~Dru\c{t}u, and L.~Mosher, \emph{{Thick metric spaces, relative
  hyperbolicity, and quasi-isometric rigidity}}, {Math. Annalen} \textbf{344}
  (2009), 543--595.

\bibitem[BDS]{BehrstockDrutuSapir:mcg}
J.~Behrstock, C.~Dru\c{t}u, and M.~Sapir, \emph{Median structures on asymptotic
  cones and homomorphisms into mapping class groups}, Proc. London Math. Soc. \textbf{102} (2011), no. 3, 503–-554.


\bibitem[Beh06]{Behrstock:asymptotic}
J.~A. Behrstock, \emph{{Asymptotic geometry of the mapping class group and
  {T}eichm\"{u}ller space}}, {Geom. Topol.} \textbf{10} (2006), 1523--1578.

\bibitem[BF]{BestvinaFujiwara:symmetric}
M.~Bestvina and K.~Fujiwara, \emph{A characterization of higher rank symmetric
  spaces via bounded cohomology}, preprint \textsc{arXiv:math/0702274}.

\bibitem[BH99]{Bridson-Haefliger}
M.~Bridson and A.~Haefliger, \emph{Metric spaces of non-positive curvature},
  Springer-Verlag, Berlin, 1999.

\bibitem[BM08]{BrockMasur:WPrelhyp}
J.~Brock and H.~Masur, \emph{Coarse and synthetic {Weil--Petersson} geometry:
  quasi-flats, geodesics, and relative hyperbolicity}, Geom. Topol. \textbf{12}
  (2008), no.~3, 2453--2495.

\bibitem[BN08]{BehrstockNeumann:qigraph}
J.~Behrstock and W.~Neumann, \emph{Quasi-isometric classification of graph
  manifold groups}, Duke Math. J. \textbf{141} (2008), no.~2, 217--240.

\bibitem[Bow08]{Bowditch:tightgeod}
B.~Bowditch, \emph{Tight geodesics in the curve complex}, Invent. Math.
  \textbf{171} (2008), no.~2, 281--300.

\bibitem[BS87]{BurnsSpatzier}
Keith Burns and Ralf Spatzier, \emph{Manifolds of nonpositive curvature and
  their buildings}, Inst. Hautes \'Etudes Sci. Publ. Math. (1987), no.~65,
  35--59.

\bibitem[DMS10]{DrutuMozesSapir}
Cornelia Dru{\c{t}}u, Shahar Mozes, and Mark Sapir, \emph{Divergence in
  lattices in semisimple {L}ie groups and graphs of groups}, Trans. Amer. Math.
  Soc. \textbf{362} (2010), no.~5, 2451--2505.

\bibitem[DR08]{Duchin-Rafi}
M.~Duchin and K.~Rafi, \emph{{Divergence of geodesics in Teichm\"uller space
  and the Mapping Class group}}, preprint {\textsc{arXiv:0611359}}, 2008.

\bibitem[DS05]{DrutuSapir:TreeGraded}
C.~Dru{\c{t}}u and M.~Sapir, \emph{Tree-graded spaces and asymptotic cones of
  groups}, Topology \textbf{44} (2005), 959--1058, {with an appendix by D. Osin
  and M. Sapir}.

\bibitem[Ger94a]{Gersten:divergence3}
S.M. Gersten, \emph{Divergence in 3-manifold groups}, Geom. Funct. Anal.
  \textbf{4} (1994), no.~6, 633--647.

\bibitem[Ger94b]{Gersten:divergence}
\bysame, \emph{Quadratic divergence of geodesics in {CAT(0)}-spaces}, Geom.
  Funct. Anal. \textbf{4} (1994), no.~1, 37--51.

\bibitem[Gro93]{Gromov:Asymptotic}
M.~Gromov, \emph{Asymptotic invariants of infinite groups}, Geometric Group
  Theory, Vol. 2 (Sussex, 1991) (G.~Niblo and M.~Roller, eds.), LMS Lecture
  Notes, vol. 182, Cambridge Univ. Press, 1993, pp.~1--295.

\bibitem[{Ham}]{Hamenstadt:announcement}
U.~{Hamenst\"adt}, \emph{{research announcement}}, 2010.

\bibitem[{Ham}09]{Hamenstadt:Linesminimaout}
\bysame, \emph{{Lines of Minima in Outer space}}, Preprint,
  {\textsc{arXiv:math/0911.3620}}, 2009.

\bibitem[JOR10]{JOR-conj}
R.~Ji, C.~Ogle, and B.~Ramsey, \emph{Relatively hyperbolic groups, rapid decay
  algebras and a generalization of the Bass conjecture}, J. Noncommut. Geom.
  \textbf{4} (2010), 83--124.

\bibitem[KL96]{KapovichLeeb:actions}
Michael Kapovich and Bernhard Leeb, \emph{Actions of discrete groups on
  nonpositively curved spaces}, Math. Ann. \textbf{306} (1996), no.~2,
  341--352.

\bibitem[KL97]{KapovichLeeb:Haken}
M.~Kapovich and B.~Leeb, \emph{Quasi-isometries preserve the geometric
  decomposition of Haken manifolds}, Invent. Math. \textbf{128} (1997), no.~2,
  393--416.

\bibitem[KL98]{KapovichLeeb:3manifolds}
\bysame, \emph{{$3$}-manifold groups and nonpositive curvature}, Geom. Funct.
  Anal. \textbf{8} (1998), no.~5, 841--852.

\bibitem[Mac]{Macura:polydiv}
Nata{\v{s}}a Macura, \emph{{CAT}(0) spaces with polynomial divergence of
  geodesics}, preprint 2011.

\bibitem[MM00]{MasurMinsky:complex2}
H.~Masur and Y.~Minsky, \emph{{Geometry of the complex of curves II:
  Hierarchical structure}}, Geom. Funct. Anal. \textbf{10} (2000), no.~4,
  902--974.

\bibitem[MSW05]{MSW:QTTwo}
L.~Mosher, M.~Sageev, and K.~Whyte, \emph{{Quasi-actions on trees II: Finite
  depth Bass-Serre trees}}, preprint, arXiv:math.GR/0405237, 2005.

\bibitem[OOS05]{OOS}
A.Yu. Olshanskii, D.~V. Osin, and M.~V. Sapir, \emph{Lacunary hyperbolic
  groups}. With an appendix by Michael Kapovich and Bruce Kleiner.
Geom. Topol. \textbf{13} (2009), no.~4, 2051--2140.

\bibitem[Osi06]{Osin:RelHyp}
D.~Osin, \emph{{Relatively hyperbolic groups: intrinsic geometry, algebraic
  properties, and algorithmic problems}}, Mem. Amer. Math. Soc. \textbf{179}
  (2006), no.~843, vi+100pp.

\bibitem[Pau01]{Pauls:NilpotentCAT(0)}
Scott Pauls, \emph{The large scale geometry in nilpotent {L}ie groups}, Commun.
  Anal. Geom. (2001), 951–982.

\bibitem[Sal]{Sale:thesis}
A.~Sale, Ph.D. thesis, University of Oxford, in preparation.

\bibitem[Ser80]{Serre:trees}
J.~P. Serre, \emph{Trees}, Springer, New York, 1980.

\bibitem[Tao11]{JTao:conj}
J.~Tao, \emph{Linearly bounded conjugator property for mapping class groups},
  preprint, \textsc{arXiv:1006.2341}, 2011.

\end{thebibliography}
\end{document}